\documentclass[reqno,11pt]{amsart}
\usepackage[english]{babel}
\usepackage{amsmath,amsfonts,amsthm,amssymb,amsbsy,upref,color,graphicx,amscd,hyperref,enumerate,bbm,comment,algorithm,algorithmic}
\usepackage{mathtools,url}
\usepackage[a4paper,margin=3.00truecm]{geometry}
\usepackage[active]{srcltx}
\usepackage[mathscr]{eucal}
\usepackage{enumitem}

\newtheorem{theorem}{Theorem}[section]
\newtheorem{lemma}[theorem]{Lemma}
\newtheorem{proposition}[theorem]{Proposition}
\newtheorem{corollary}[theorem]{Corollary}
\newtheorem{definition}[theorem]{Definition}
\theoremstyle{remark}
\newtheorem{remark}[theorem]{Remark}
\newtheorem{example}[theorem]{Example}

\numberwithin{equation}{section}

\makeatletter
\newcommand{\proofpart}[2]{%
  \par
  \addvspace{\medskipamount}%
  \noindent\emph{Part #1: #2}\par\nobreak
  \addvspace{\smallskipamount}%
  \@afterheading
}
\makeatother

\newcommand{\smallO}[1]{\scriptstyle\mathcal{O}}

\newcommand{\restr}{\mathbin{\vrule height 1.6ex depth 0pt width 0.13ex\vrule height 0.13ex depth 0pt width 1.3ex}}
\newcommand{\bib}[4]{\bibitem{#1}{\sc#2: }{\it#3. }{#4.}}

\newcommand{\R}{\mathbb R}

\newcommand{\N}{\mathbb N}

\newcommand{\Om}{\Omega}

\newcommand{\cA}{\mathcal A}

\newcommand{\cH}{\mathcal H}
\newcommand{\cF}{\mathcal F}

\newcommand{\cJ}{\mathcal J}
\newcommand{\eps}{\varepsilon}

\newcommand{\divs}{\mathrm{div}}

\newcommand{\cL}{\mathcal L}
\newcommand{\be}{\begin{equation}}
\newcommand{\ee}{\end{equation}}
\newcommand{\X}{\mathfrak X}

\newcommand{\per}{\mathrm{Per}}

\newcommand{\sbv}{\mathrm{SBV}}
\newcommand{\gsbv}{\mathrm{GSBV}}
\newcommand{\bv}{\mathrm{BV}}

\def\Xint#1{\mathchoice
{\XXint\displaystyle\textstyle{#1}}%
{\XXint\textstyle\scriptstyle{#1}}%
{\XXint\scriptstyle\scriptscriptstyle{#1}}%
{\XXint\scriptscriptstyle\scriptscriptstyle{#1}}%
\!\int}
\def\XXint#1#2#3{{\setbox0=\hbox{$#1{#2#3}{\int}$ }
\vcenter{\hbox{$#2#3$ }}\kern-.58\wd0}}

\def\dashint{\Xint-}

\newcommand{\brmk}{\begin{remark}}
\newcommand{\ermk}{\end{remark}}
\newcommand{\bd}{\begin{definition}}
\newcommand{\ed}{\end{definition}}
\newcommand{\bt}{\begin{theorem}}
\newcommand{\et}{\end{theorem}}
\newcommand{\bl}{\begin{lemma}}
\newcommand{\el}{\end{lemma}}
\newcommand{\bpr}{\begin{proposition}}
\newcommand{\epr}{\end{proposition}}
\newcommand{\bex}{\begin{example}}
\newcommand{\eex}{\end{example}}

\title{Shape optimization problems for functionals with a boundary integral}

\author{Giuseppe Buttazzo and Francesco Paolo Maiale}

\begin{document}

\maketitle
{\it Dedicated to Umberto Mosco for his 80th birthday}

\begin{abstract}
We consider shape optimization problems for general integral functionals of the calculus of variations that may contain a boundary term. In particular, this class includes optimization problems governed by elliptic equations with a Robin condition on the free boundary. We show the existence of an optimal domain under rather general assumptions and we study the cases when the optimal domains are open sets and have a finite perimeter.
\end{abstract}


%%%%%%%%%%%%%%%%%%%%%%%%%%%%%%%%%%%%%%%%%%%%%%%%%%%%%%%%%%%%
\section{Introduction}\label{sec:intro}

Let $D \subset \R^d$ be a bounded open set with a Lipschitz boundary. In this paper, we are interested in the following {\em shape optimization problem}.

%% Per cambiare l'indentazione del problema basta variare .08 e .92 in modo che la somma sia sempre uno.
\medskip
\noindent\hspace{.08\linewidth} \begin{minipage}{.92\textwidth}
{\bf Problem.} Find a domain $\Om$ which solves the minimization problem
$$\inf\big\{\cJ(\Om)\ :\ \Om\subset D,\ \Om\text{ Lipschitz}\big\},$$
where the {\em shape functional} $\cJ$ is defined by
$$\cJ(\Om):=\min\left\{\int_\Om j(x,u,\nabla u)\,dx+\int_{\partial\Om} g(x,u)\,d\cH^{d-1}\ :\ u\in W^{1,p}(\Om)\right\}.$$
Here $p>1$, $\cH^{d-1}$ is the $(d-1)$-dimensional Hausdorff measure, and the integrands $j$ and $g$ satisfy suitable properties.
\end{minipage}
\medskip

\noindent The prototype of our class of integral functionals can be obtained by solving the PDE with {\it Robin boundary condition}
\be\label{robinpde}
\begin{cases}
-\Delta u=f&\text{in $\Om$}\\
\beta u+\partial_\nu u=0&\text{on $\partial\Om$}
\end{cases}\ee
and minimizing the corresponding energy
$$\frac12\int_\Om|\nabla u|^2\,dx-\int_\Om fu\,dx+\frac\beta2\int_{\partial\Om}u^2\,d\cH^{d-1},$$
among all domains $\Om\subset D$ with prescribed Lebesgue measure $|\Om|$. This corresponds to the integrands
$$j(x,s,z)=\frac12|z|^2-f(x)s+c,\qquad g(x,s)=\frac\beta2 s^2,$$
where $c$ is the Lagrange multiplier associated to the measure constraint on $\Om$. The stability of solutions of elliptic equations under Robin boundary conditions, with respect to the variation of the domain, has been studied in \cite{bugitr16}.

When the Robin boundary condition is replaced by the Dirichlet condition
$$u=0\hbox{ on }\partial\Om,$$
and $g(x,0)=0$, the boundary integral disappears, and the corresponding shape optimization problem has been considered in \cite{butshr}. The shape optimization problem under Robin boundary condition on $\partial\Om$, in a general form, was first studied by Bucur-Giacomini in \cite{bugi16}. They considered the right-hand side $f=0$ but with a Dirichlet condition of the form $u=u_0$ on a prescribed part $D_0$ of $D$. The addition of a zero-order term in our framework is the natural variational formulation of problem \eqref{robinpde} and raises several technical difficulties.

The key idea, introduced in \cite{bugi16}, consists in extending all functions $u\in W^{1,p}(\Om)$ to the whole $\R^d$ by setting $u=0$ outside $\Om$. In this way, the Sobolev property of the extension is clearly lost; however, the extended functions belong to some $\sbv$ space, which allows us to write the boundary integral corresponding to a function $u$ as an integral on the {\it jump set} of $u$. In this way, the dependence on the domain $\Om$ can be removed, and the problem is reduced to the minimization of a functional of the calculus of variations which includes a standard integral term and an additional integral term over the jump set of the competing functions $u$.

The main result (see Section \ref{sec:results} for more details) of our paper is the existence of an optimal shape $\Om$ which is open and has finite perimeter. To achieve this result we first consider the relaxation of our initial problem to the family of sets
$$\cA(D):=\left\{\Om\subset D\ :\ \text{$\Om$ open, $\partial\Om$ is $\cH^{d-1}$-rectifiable and $\cH^{d-1}(\partial\Om)<\infty$}\right\}$$
by setting for $\Om\in\cA(D)$
\be\label{functionalrelaxed}
\cJ(\Om):=\inf_{u\in W^{1,p}(\Om)}\left\{\int_\Om j(x,u,\nabla u)\,dx+\int_{\partial\Om}\left[g(x,u^+)+g(x,u^-)\right]\,d\cH^{d-1}\right\},
\ee
where $\nabla u$, $u^+$, $u^-$ will be defined in Section \ref{sec:setting}. Next, we show that the shape optimization problem can be reformulated in terms of the following {\it free discontinuity functional}
\be\label{eq.intro2}
\cF(u)=\int_{\{u\neq0\}} j(x,u,\nabla u)\,dx+\int_{J_u}\left[g(x,u^+)+g(x,u^-)\right]\,d\cH^{d-1},
\ee
defined on the set of functions
$$\mathscr{F}_D:=\left\{u\in\sbv(\R^d)\ :\ u=0\text{ on }\R^d\setminus D\right\},$$
where $J_u$ denotes the jump set of $u$. The paper is organized as follows:
\begin{enumerate}
\item The relaxed shape optimization problem reduces to a problem of the calculus of variations; more precisely, we consider
\be\label{eq.intro3}
\min\big\{\cF(u)\ :\ u\in\mathscr{F}_D\big\}.
\ee
Indeed, if $\bar{\Om}$ is an optimal shape for the minimization problem associated to the relaxed functional \eqref{functionalrelaxed}, then $\bar{\Om}$ coincides (see Lemma \ref{lemma.reduction}) with the set $\{\bar{u}\ne0\}$, where $\bar{u}$ is a solution of the variational problem \eqref{eq.intro3}.

\item The functional in \eqref{eq.intro3} is not coercive on $\mathscr{F}_D$; thus, to obtain lower semicontinuity and compactness, we extend the functional to the functional space
$$\X_D:=\big\{u\ :\ u\vee\eps,\ u\wedge(-\eps)\in\gsbv(\R^d)\ \forall\eps>0,\text{ $u=0$ on $\R^d\setminus D$}\big\},$$
where $\gsbv(\R^d)$ is the space of all functions $u$ such that
$$u\wedge M\text{ and }u\vee(-M)\text{ are in }SBV(\R^d)\qquad\text{for every }M>0.$$
In this extended framework, we prove (see Theorem \ref{thm.main.1}) that, under some assumptions on the integrands $j$ and $g$, the minimization problem
$$\min\big\{\cF(u)\ :\ u\in\X_D\big\}$$
admits a solution $\bar{u}\in\X_D$ which, a priori, may not belong to $\sbv(\R^d)$.

\item The next step is achieved as a consequence of Theorem \ref{boundbelow}, which asserts that
$$|\bar{u}|\ge\alpha>0\qquad\text{where }\bar{u}\ne0,$$
and, as an immediate consequence, that $\bar{u}\in\gsbv(\R^d)$. This solution is then proved to be more regular, i.e., 
$$\bar{u}\in\sbv(\R^d)\quad\text{with $\bar{u}=0$ on $\R^d\setminus D$}$$
as a corollary of Lemma \ref{lemma.boundinf}, which tells us that $\bar{u} \in L^\infty(\R^d)$. In particular (see Theorem \ref{thm.main.2}), the optimal shape $\bar{\Om}$ has finite perimeter.

\item Using an argument concerning the Mumford-Shah-type functional
$$MS(u):=\int_{\R^d} f(x,\nabla u)\,dx+\cH^{d-1}(J_u)$$
in Theorem \ref{thm.main.3} we prove that the optimal shape $\bar{\Om}=\{\bar{u}\ne0\}$ is open (and thus belongs to $\cA(D)$) and there holds
$$\cJ(\bar{\Om})=\inf\big\{\cJ(\Om)\ :\ \Om\in\cA(D),\ \Om\text{ Lipschitz}\big\}.$$
\end{enumerate}

%%%%%%%%%%%%%%%%%%%%%%%%%%%%%%%%%%%%%%%%%%%%%%%%%%%%%%%%%%%%
\section{Preliminaries and main results}\label{sec:setting}

Throughout this paper, we indicate by $|\cdot|$ the $d$-dimensional Lebesgue measure and by $\cH^{d-1}$ the $(d-1)$-dimensional Hausdorff measure. We use $\chi_E$ to indicate the characteristic function of a subset $E$ of $\R^d$, defined by
$$\chi_E(x) = \begin{cases} 1 & \text{if } x\in E,\\ 0 & \text{if } x\notin E. \end{cases}$$
We denote by $Du$ the distributional derivative of a function $u$ and by $B_r(x)$ the open ball centered at $x$ with radius $r$.

%%%%%%%%%%%%%%%%%%%%%%%%%
\subsection{Lower semicontinuity and compactness in $\sbv(\R^d)$} \label{sec:preliminaryresults}

The natural framework of our paper is the $\sbv$ space, the class of {\em special functions of bounded variation} which was introduced by De Giorgi and Ambrosio in \cite{degamb1} and used by De Giorgi, Carriero and Leaci in \cite{dgcl89} to successfully solve the {free discontinuity minimum problems}. Before giving the formal definition, we recall some well-known notions and we refer to \cite{ambrosio2} and to the book \cite{afp00} for all details.

\bd
Let $A\subset\R^d$ be an open set. Given a function $u\in L^1(A)$, the {\bf total variation} of $u$ in $A$ is defined as
$$\int_A|Du|:=\sup\left\{\int_A u\,\divs(\phi)\,dx\ :\ \phi\in C_c^1(A,\R^d),\ \|\phi\|_\infty\le1\right\}.$$
The space $\bv(A)$ of all functions with bounded total variation in $A$ is then
$$\bv(A):=\left\{u\in L^1(A)\ :\ \int_A|Du|<+\infty\right\}.$$
In other words, a function $u\in L^1(A)$ belongs to $\bv(A)$ if and only if its distributional derivative $Du$ belongs to the space of finite vector-valued Radon measures.
\ed

Given a function $u:\R^d\to\R$ the {\em precise representative} of $u$, which belongs to the same class in $L^1(\R^d)$, is defined by setting
$$\mathfrak{u}(x):=\lim_{r\to0}\dashint_{B_r(x)}u(y)\,dy.$$
If $u\in W^{1,p}(\R^d)$, $1<p<\infty$, the limit above exists up to a set of $p$-capacity zero (shortly $C_p$-a.e., where $C_p$ denotes the $p$-capacity), while if $u\in\bv(\R^d)$ it exists up to a set of $\cH^{d-1}$ measure zero. If $u\in\bv(A)$ we denote by $\nabla u$ the part of the measure $Du$ which is absolutely continuous with respect to the Lebesgue measure, and by $D^su$ the corresponding singular part. The asymptotic values of $u$ near a point of discontinuity, are defined by
\[\begin{split}
&u^-(x)=\sup\bigg\{t\in\R\ :\ \lim_{r\to0}\frac{\big|B_r(x)\cap\{u<t\}\big|}{|B_r(x)|}=0\bigg\},\\
&u^+(x)=\inf\bigg\{t\in\R\ :\ \lim_{r\to0}\frac{\big|B_r(x)\cap\{u>t\}\big|}{|B_r(x)|}=0\bigg\},
\end{split}\]
and the {\it jump set} $J_u$ is the set where $u^+\ne u^-$.
\bd 
Let $A\subset\R^d$ be an open set. The class $\sbv(A)$ of special functions with bounded variation on $A$ consists of all $u\in\bv(A)$ such that the total variation of the singular measure $D^su$ is given by
$$|D^su|=|u^+-u^-|\cH^{d-1}\restr J_u.$$
%\left\{ x \in A \: : \: u(x) \neq \frac{u(x^+)+u(x^-)}{2} \right\}$$
%or, in other words, such that the singular part of the measure $Du$ is concentrated on the set of points where the {\em precise representative} of $u$ is not defined.
\ed

As mentioned in the introduction, the space $\sbv(\R^d)$ is the natural one to solve the minimization problem \eqref{eq.intro3}. However, the coercivity in $\sbv(\R^d)$ of the functional $\cF$ in \eqref{eq.intro2} is not guaranteed because pathological behaviors are, in principle, still possible.

\bd
Let $A\subset\R^d$ be an open set. The class $\gsbv(A)$ is defined as the set of functions $u$ such that
$$u\wedge M\text{ and }u\vee(-M)\text{ are in }SBV(A)\qquad\text{for every }M>0.$$
In other words, a function belongs to $\gsbv(A)$ if and only if any truncation that makes the $L^\infty$-norm finite is an element of $\sbv(A)$.
\ed

Following the strategy proposed in \cite{bucgia1} and taking into account that we do not have a natural constraint on the $L^\infty$-norm, we need first to work on the space
$$\X_D:=\big\{ u \ :\ u\vee\eps,\ u\wedge(-\eps)\in\gsbv(\R^d)\ \forall\eps>0,\text{ $u = 0$ on $\R^d\setminus D$}\big\}.$$
In this way we may avoid wild oscillations near $\{u=0\}$ and, at the same time, we may work with the bounded truncations $u\wedge M$, $u\vee(-M)$. Going back to the main framework, we need to introduce the weak convergence in $\sbv(\R^d) \cap L^\infty(\R^d)$.

\bd\label{rmk.weakconvergence}
Let $A\subset\R^d$ be an open set. A sequence $\{u_n\}_{n \in \N}$ converges to $u$ weakly in $\sbv(A) \cap L^\infty(A)$ if the following properties are satisfied: \mbox{}
\begin{enumerate}[itemsep=.3em]
\item $u_n(x) \to u(x)$ at a.e. $x \in A$, 
\item $\nabla u_n \rightharpoonup \nabla u$ weakly in $L^1(A)$,
\item both $\|u_n\|_{\infty}$ and $\cH^{d-1}(J_{u_n})$ are uniformly bounded.
\end{enumerate}
\ed

The lower semicontinuity of functionals defined in $\sbv(A)$ was first studied by Ambrosio in \cite{ambrosio1} and will be crucial in the following.

\bt[Ambrosio]\label{lsc.sbv}
Let $\varphi(x,s,z)$ be a Carath\'eodory function on $\R^d\times\R\times\R^d$ and let $\psi(x,a,b)$ be a continuous function on $\R^d\times\R\times\R$. Suppose that:
\begin{enumerate}[label=(\roman*)]
\item the function $\varphi(x,s,\cdot)$ is convex;
\item there is $r>1$ such that the estimate
$$\varphi(x,s,z)\ge|z|^r$$
holds true for all $z\in\R^d$, all $s\in\R$ and a.e. $x\in\R^d$;
\item the function $\psi$ is nonnegative and satisfies the triangular inequality, namely
$$\psi(x,a,b)\le\psi(x,a,c)+\psi(x,c,b).$$
\end{enumerate}
Then for every open set $A\subset\R^d$ the functional
$$F(u)=\int_A \varphi(x,u,\nabla u)\,dx+\int_{J_u}\psi(x,u^+, u^-)\,d\cH^{d-1}$$
is lower semicontinuous in $\sbv(A)$ with respect to the $L_{\mathrm{loc}}^1(A)$ convergence.
\et

\brmk\label{rmk.triangular}
In \cite{ambrosio1} the result above was actually obtained under weaker assumptions on $\varphi$ and $\psi$. However, in our case, we have
$$\beta_1 \left[ g(u^+) + g(u^-) \right] \le \psi(x,u^+,u^-) \le \beta_2 \left[ g(u^+) + g(u^-) \right],$$
with $\beta_1,\beta_2$ positive, and $g$ positive; hence the triangular inequality is automatically satisfied in view of the assumptions in Section \ref{sec:hp}.
\ermk

\brmk\label{rmk.lsclsc}
In Theorem \ref{lsc.sbv}, the assumption “$\varphi(x,\cdot,\cdot)$ continuous for a.e. $x$” can be weakened. Indeed, using an approximation argument and Beppo Levi's monotone convergence theorem, one can require $\varphi(x,\cdot,\cdot)$ to be lower semicontinuous only.
\ermk

\brmk\label{rmk.lscingsbv}
Note that this result also applies to sequences $u_n$ in $\gsbv(A)$. Indeed, using the same notation, we can write the functional as
$$F(u)=\int_{A\cap\{|u|<M\}}\!\!\!\varphi(x,u,\nabla u)\,dx+\int_{J_u\cap\{|u| < M\}}\!\!\!\psi(x,u^+,u^-)\,d\cH^{d-1}+o(1)=F_1(u)+o(1)$$
which means that $F_1$ is lower semicontinuous in $\sbv(A)$ with respect to the $L_{\mathrm{loc}}^1(A)$ topology. In particular, we have
$$
\liminf_n F_1(u_n) \ge F_1(u)
$$
so that, by taking the limit as $M \to \infty$, we easily deduce that the same is true with $F$ in place of $F_1$.
\ermk

We conclude this section with a few results giving the coercivity of functionals in $\sbv(A)$ and $\gsbv(A)$. The first one was proved by Ambrosio in \cite{ambrosio2}.

\bt[Ambrosio] \label{compactness.sbv}
Let $A$ be an open bounded set in $\R^d$, let $\phi : [0,\infty)\to[0,\infty]$ be a convex non-decreasing function satisfying the condition
$$\lim_{t \to \infty} \frac{\phi(t)}{t}=\infty,$$
and let $\Theta : [0,\infty]\to[0,\infty]$ be a concave non-decreasing function such that
$$\lim_{t \to 0^+} \frac{\Theta(t)}{t}=\infty.$$ 
Let $(u_n)_{n\in\N}\subset\sbv(A)\cap L^\infty(A)$ be a sequence such that $\|u_n\|_{L^\infty}\le C$ for a suitable constant $C$, and
\be \label{conditiona1}
\sup_{n \in \N}\left\{ \int_A \phi(|\nabla u_n|) \, d x + \int_{J_{u_n}} \Theta(|u_n^+ - u_n^-|) \, d \cH^{d-1} \right\} < \infty.
\ee
Then there exists a subsequence converging in measure to a function $u\in\sbv(A)\cap L^\infty(A)$ such that
$$\nabla u_{h_k}\rightharpoonup\nabla u\qquad\hbox{ weakly in }L^1(A).$$
\et

\brmk
\label{cpt.ms}
If we take the concave non-decreasing function
\[
\Theta(t) := \begin{cases} 0 & \text{if $t = 0$}, \\ 1 & \text{if $t > 0$}, \end{cases}
\]
and the convex non-decreasing function $\phi(t) = |t|^q$, then condition \eqref{conditiona1} reads as
\be \label{conditiona2}
\sup_{n \in \N}\left\{ \int_A |\nabla u_n|^q \, dx + \cH^{d-1}(J_{u_n}) \right\} < \infty.
\ee
The functional appearing in \eqref{conditiona2} is usually known as the {\em Mumford-Shah functional} and we will use it again to show that optimal domains $\Om$ are open.
\ermk

\bl \label{lemma.aux1}
Let $(u_n)_{n\in\N}\subset\gsbv(A)$ be a sequence of functions and assume that the following properties hold:
\begin{enumerate}
\item there exist $c_1,c_2>0$ and $q>1$ such that $\|u_n\|_{L^q(A)}\le c_1$ and $\|\nabla u_n\|_{L^q(A)}\le c_2$;
\item there exists $c_3>0$ such that $\int_{J_{u_n}} \left[ |u_n^+|^q + |u_n^-|^q \right] \,d \cH^{d-1} \le c_3$;
\item there exists $c_4>0$ such that $\cH^{d-1}(J_{u_n})\le c_4$.
\end{enumerate}
Then $u_n$ converges, up to subsequences, strongly in $L^q(A)$ to a function $u\in\gsbv(A)$ that satisfies the properties $(1)$, $(2)$ and $(3)$ with the same constants above.
\el

\begin{proof}
Using $(1)$ we immediately deduce that $u_n$ converges up to subsequences to some $u$ weakly in $L^q(A)$ which satisfies the inequality
$$
\| u \|_{L^q(A)} \le \liminf_{n \to \infty} \|u_n\|_{L^q(A)} \le c_1.
$$
Therefore, using a standard result in functional analysis, to prove that $u_n$ converges strongly to $u$ in $L^q(A)$ it is enough to prove the convergence of the norms:
$$
\int_A |u_n|^q \, dx \xrightarrow{n \to \infty} \int_A |u|^q \, dx.
$$
Now notice that the weak derivative of $u_n$ can be written as the sum of the absolutely continuous part and the singular one; namely, we have
$$
D u_n = \nabla u_n \cdot dx + \left(u_n^+ - u_n^-\right) \nu_n \cdot d\cH^{d-1} \restr J_{u_n},
$$
where $\nu_n$ is the normal unit vector to $J_{u_n}$. Let $w_n := |u_n|^q$. A simple computation shows that for its weak derivative we have
$$
|D w_n| = q |u_n|^{q-1} |\nabla u_n| \cdot dx + \left| |u_n^+|^q - |u_n^-|^q \right| \cdot d\cH^{d-1} \restr J_{u_n}
$$
so that its total variation as a measure is given by
\[ \begin{aligned}
\int_A |D w_n| & = q \int_A |u_n|^{q-1} |\nabla u_n| \, dx + \int_{J_{u_n}} \left| |u_n^+|^q - |u_n^-|^q \right| \, d\cH^{d-1}
\\ & \le q \int_A \left[ |u_n|^q + |\nabla u_n|^q \right] \, dx + \int_{J_{u_n}} \left[ |u_n^+|^q + |u_n^-|^q \right] \,d\cH^{d-1}
\\ & \le q \left[ c_1^q + c_2^q \right] + c_3 =: C_q.
\end{aligned} \]
This means that $w_n$ converges to some $w$ in $\bv(A)$ and, in particular, it converges strongly in $L^1(A)$ to $w$. Finally, we notice that
\[\begin{aligned}
\int_A |D u_n| & = \int_A |\nabla u_n| \, dx + \int_{J_{u_n}} \left| u_n^+ - u_n^- \right| \, d\cH^{d-1} 
\\ & \le \int_A \left(1+|\nabla u_n|^q\right) \, dx + \int_{J_{u_n}} \left(1+|u_n^+|^q + |u_n^-|^q\right) \, d\cH^{d-1}
\\ & \le |A| + \cH^{d-1}(J_{u_n}) + \int_{A} |\nabla u_n|^q \, dx + \int_{J_{u_n}} \left[ |u_n^+|^q + |u_n^-|^q \right] \,d \cH^{d-1}
\\ & \le |A| + c_4 + c_2^q + c_3 =: \tilde{C}_{q,|A|},
\end{aligned}\]
which means that $u_n$ converges strongly to $u$ in $L^1(A)$ and weakly in $L^q(A)$. Therefore, using the fact that
$$
\int_A |u_n|^q \, dx = \int_A |w_n| \, dx \xrightarrow{n \to \infty} \int_A |w| \, dx,
$$
we easily deduce that $|u|^q=w$, concluding the proof of the result.
\end{proof}

%%%%%%%%%%%%%%%%%%%%%%%%%
\subsection{Assumptions on $j$ and $g$}
\label{sec:hp}

We will now briefly describe the assumptions on $j$ and $g$ that are sufficient to obtain our main results. The model integrands are respectively
$$
j(x,u,\nabla u) = |\nabla u|^p - f(x)u + 1,
$$
where $f$ is a function that belongs to a Lebesgue space, and
$$
g(x,u) = \beta |u|^q
$$
for some $\beta > 0$ so that for $p = q = 2$ we obtain the Robin-type problem \eqref{robinpde}. In particular, we require that $j$ satisfies some (or all) of the following properties: \mbox{}
\begin{enumerate}[label=(j\arabic*)]
\item $x\mapsto j(x,s,\xi)$ is measurable for all $(s,\xi) \in \R \times \R^d$, $(s,\xi)\mapsto j(x,s,\xi)$ is lower semicontinuous for almost every $x \in\R^d$ and $\xi \mapsto j(x,s,\xi)$ is convex for all $s \in \R$ and almost every $x \in \R^d$.
\item The function $j(x,0,0)$ belongs to $L^1(\R^d)$ and for a.e. $x \in \R^d$ we have $j(x,0,0)\ge0$.
\item There are $p>1$ and $L>0$ such that
\be\label{eq.f.1}
j(x,s,\xi)-j(x,s,0)\ge L |\xi|^p,
\ee
and functions $f\in L^\infty(\R^d)$, $a\in L^1(\R^d)$, and $1<q\le p$ for which
\be\label{eq.f.2}
j(x,s,0) \ge-f(x)|s|^q - a(x),
\ee
holds for a.e. $x\in\R^d$. Furthermore, the function $f$ satisfies the estimate
\be \label{eq.f.2.2}
\|f\|_{L^\infty(\R^d)} \le \frac{L}{2q} \lambda_{\beta_1/L,q}(B_D),
\ee
where $\beta_1$ is introduced in assumption (g3), $B_D$ is any ball of volume $|D|$ and for $\alpha>1$ and $b>0$ we define
$$
\lambda_{b,\alpha}(B) := \min_{u \in W^{1,\alpha}(B) \setminus\{0\}} \frac{ \int_B |\nabla u|^\alpha \, dx + b \int_{\partial B} |u|^\alpha \, d\cH^{d-1}}{ \int_ B |u|^\alpha \, dx}.
$$
\item There exists $\eps_0>0$ such that for a.e. $x \in \R^d$ there holds
$$j(x,s,0) - j(x,t,0) \ge 0 \quad \text{for all $s < t < \eps_0$}.$$
Furthermore, if $q$ is the exponent given in (j3), then
\be\label{eq.f.3}
p\ge q>\max \left\{1,\frac{p}{2p-1}\left[p+\frac{(p-1)^2}{(d-1)p}\,\frac{2}{1+\sqrt{1+\frac{4(p-1)}{(d-1)p}}}\right] \right\}.
\ee
\item The lower bound \eqref{eq.f.1} is an equality, namely
$$j(x,s,\xi) - j(x,s,0) = L |\xi|^p.$$
Furthermore, there are $M_0 > 0$ large and a positive constant $C_j$ such that
\be \label{eq.f.5}
j(x,s,0) - j(x,t,0) \ge - C_j |s|^q \quad \text{for all $s \ge t > M_0$}.
\ee
\end{enumerate}

In a similar fashion, we require that $g$ satisfies some (or all) of the following properties:
\begin{enumerate}[label=(g\arabic*)]
\item $x \mapsto g(x,s)$ is measurable for all $s\in \R$ and $s \mapsto g(x,s)$ is lower semicontinuous for a.e. $x\in \R^d$.
\item For every $x\in \R^d$ we have $g(x,0)=0$.
\item There exists a continuous positive function $\beta_1 : \R^d \to (0,\infty)$ such that
\be\label{eq.g.1}
g(x,s)\ge\beta_1(x)|s|^q
\ee
for a.e. $x\in\R^d$ and all $s\in\R$, where $q$ is as in (j3). Furthermore, $\beta_1:=\min_{x\in\R^d}\beta_1(x)>0$.
\item There exists a continuous positive function $\beta_2 : \R^d \to [0,\infty)$ such that
$$\beta_2(x) |s|^q \ge g(x,s) \ge \beta_1(x) |s|^q$$
for a.e. $x\in\R^d$ and all $s\in\R$, where $q$ is as in (j3). Furthermore, $\beta_2:=\max_{x\in\R^d}\beta_2(x)>\beta_1$.
\end{enumerate}

%%%%%%%%%%%%%%%%%%%%%%%%%
\subsection{Main results}\label{sec:results}

The first step, the existence of a solution $\bar{u}$ for the minimization problem \eqref{eq.intro3} with $\X_D$ in place of $\mathscr{F}_D$, is obtained under very mild assumptions:

\bt \label{thm.main.1}
Suppose that $j$ and $g$ satisfy (j1)--(j3) and (g1)--(g3) respectively. Then the minimization problem
$$\min\big\{\cF(u)\ :\ u\in\X_D\big\}$$
admits a solution $\bar{u} \in \X_D$. Furthermore, the relaxed shape optimization problem associated to the functional \eqref{functionalrelaxed} is solved by $\bar{\Om} := \{ \bar{u} \neq 0 \}$ on the class
$$\left\{\Om\subset D\ :\ \text{$\Om$ measurable}\right\}\supset\cA(D).$$
\et

The proof of the second assertion is an easy consequence of Lemma \ref{lemma.reduction}. However, the solution $\bar{u}$ only belongs to $\X_D$ so we need an additional effort to show that $\bar{u}\in\mathscr{F}_D$ or, in other words, that $\bar{u}$ has finite $L^\infty$-norm (see Lemma \ref{lemma.boundinf}) and also that it is bounded from below by a positive constant (see Theorem \ref{boundbelow}).

\bt \label{thm.main.2}
Let $j$ and $g$ be as in Theorem \ref{thm.main.1}. If we further assume (j4) and (g4), then $\bar{u}$ solves \eqref{eq.intro3} and $\bar{u}\ge0$. In other words, $\bar{u} \in \mathscr{F}_D$ and, consequently, $\per(\bar{\Om})<\infty$, so that the optimal shape $\bar{\Om}=\{\bar{u}>0\}$ belongs to the class
$$
\left\{\Om\subset D\ :\ \text{$\partial\Om$ is $\cH^{d-1}$-rectifiable with $\cH^{d-1}(\partial \Om) < \infty$} \right\} \supset \cA(D).
$$
\et

Finally, following the strategy in \cite{bugi16}, we prove that $\bar{\Om}$ is open and therefore it minimizes $\cJ$ on the class $\cA(D)$ which we recall to be defined as
$$\cA(D) := \left\{\Om\subset D\ :\ \text{$\Om$ open, $\partial \Om$ is $\cH^{d-1}$-rectifiable with $\cH^{d-1}(\partial \Om)<\infty$}\right\}.$$
The key point here is the essential closedness of the jump set associated to the optimal function $\bar{u}$ and, obviously, having $\|\bar{u}\|_\infty$ finite plays a fundamental role here.

\bt \label{thm.main.3}
Let $j$ and $g$ be as in Theorem \ref{thm.main.1}. If we further assume (j4), (g4), and (j5), then the optimal set $\bar{\Om}=\{\bar{u}>0\}$ is open. Moreover, it turns out that
$$\cJ(\bar{\Om})=\inf\big\{\cJ(\Om)\ :\ \Om \in \cA(D),\text{ $\Om$ Lipschitz}\big\}.$$
\et

%%%%%%%%%%%%%%%%%%%%%%%%%%%%%%%%%%%%%%%%%%%%%%%%%%%%%%%%%%%%
\section{Proof of the main results}

The goal of this section is to give a proof of the main results and, at the same time, introduce all the technical tools we need to carry out our analysis. The main references here are \cite{butshr} for the existence of optimal solutions in the case of Dirichlet boundary conditions, and \cite{bucgia1,bugi16} for the other properties and a Poincar\'e-type inequality.

%%%%%%%%%%%%%%%%%%%%%%%%%
\subsection{Reduction to the auxiliary problem}

In order to address the minimization of the functional \eqref{functionalrelaxed}, we now consider the already mentioned free discontinuity functional
$$
\cF(u):= \int_{\{u\neq 0\}} j(x,u,\nabla u) \, dx + \int_{J_u} \left[ g(x,u^+) + g(x,u^-) \right] \, d\cH^{d-1}.
$$
Our goal is to show that the minimization problem
$$\min\big\{\cF(u)\ :\ u\in\mathscr{F}_D\big\}$$
is, under mild assumptions on $j$ and $g$, equivalent to the minimization of $\cJ$ among all $\Om$ in $\cA(D)$ in such a way that the following holds:
$$\bar{u}=\mathrm{argmin}_{u\in\mathscr{F}_D}\cF(u)\ \implies\ \min\big\{\cJ(\Om)\ :\ \Om\in\cA(D)\big\} = \cJ( \{ \bar{u} \neq 0\} ).
$$
The reason we introduce another functional is that it is much easier to deal with the minimization with respect to functions rather than sets.

\bl \label{lemma.reduction}
Assume that $j$ and $g$ satisfy (j2) and (g2). If $\bar{u}$ is a minimizer of the auxiliary functional given in \eqref{eq.intro2}, which we recall to be
$$
\cF(u):= \int_{\{u \neq 0\}} j(x,u,\nabla u) \,dx + \int_{J_u} \left[ g(x,u^+) + g(x,u^-) \right] \, d\cH^{d-1},
$$
then the shape optimization problem
$$\min\big\{\cJ(\Om)\ :\ \Om\in\cA(D)\big\}$$
admits a solution of the form $\bar{\Om}=\{\bar{u}\ne0\}$.
\el

\begin{proof}
Let $u \in W^{1,p}(\Om)$ and extend it to zero outside of $\Om$. It is easy to verify that
\[ \begin{aligned}
\int_\Om j(x,u,\nabla u) \, dx & = \int_{\R^d} j(x,u,\nabla u)\, dx - \int_{\R^d \setminus \Om} j(x,0,0) \, dx
\\ & = \int_{\R^d} j(x,u,\nabla u) \, dx + \int_\Om j(x,0,0) \, dx - \int_{\R^d} j(x,0,0) \, dx
\\ & = \int_{\{u\neq 0\}} j(x,u,\nabla u) \, dx + \int_\Om j(x,0,0) \, dx
\\ & \ge \int_{\R^d} j(x,u,\nabla u) \chi_{\{u\neq 0\}} \, dx
\end{aligned} \]
since $j(x,0,0) \ge 0$ by assumption (j2). In a similar fashion, one has
\[\begin{aligned}
\int_{\partial \Om} g(x,u) \, d\cH^{d-1} & = \int_{J_u} \left[g(x,u^+)+g(x,0)\right] \, d\cH^{d-1} - \int_{\partial \Om} g(x,0) \, d\cH^{d-1}
\\ & = \int_{J_u} g(x,u^+) \, d\cH^{d-1}
\end{aligned} \]
as a consequence of assumption (g2). If $u$ is not regular enough, we obviously replace the latter integral as in the formula for $\cF$, namely
$$
\int_{J_u} \left[ g(x,u^+) + g(x,u^-) \right] \, d\cH^{d-1},
$$
since $u^-$ might not be equal to zero. Now, if $\bar{u}$ is a minimizer for $\cF$, setting $\bar{\Om}:=\{\bar{u}\ne0\}$, the inequalities above give for every $\Om\in\cA(D)$
$$\cJ(\Om)\ge\cF(\bar{u})\ge\cJ(\bar{\Om}),$$
which concludes the proof.
%Now let $\bar{u}$ be a minimizer for $\cF$ and notice that for any $\Om \in \cA(D)$ there is $u_\Om$ such that
%$$
%\cJ(\Om) = \cJ(\Om,u_\Om).
%$$
%Using the estimates obtained above, together with the minimality of $\bar{u}$, we get
%$$
%\cJ(\Om) \ge \cF(u_\Om) \ge \cF(\bar{u}) \ge \cJ(\bar{\Om}),
%$$
%hence
%$$\min\big\{\cJ(\Om)\ :\ \Om\in\cA(D)\big\}=\cJ( \bar{\Om} ) \qquad \text{with $\bar{\Om} := \{ \bar{u}\neq 0\}$,}$$
%which concludes the proof.
\end{proof}

%%%%%%%%%%%%%%%%%%%%%%%%%
\subsection{Proof of Theorem \ref{thm.main.1}: existence of a minimizer}

We now prove that the auxiliary functional $\cF$ has a minimizer $\bar{u}$ in the class $\X_D$, namely that the problem
$$\min\big\{\cF(u)\ :\ u\in\X_D\big\}$$
admits a solution $\bar{u}$. However, both the $\sbv$-regularity and the $L^\infty$-regularity of $\bar{u}$ are unclear at this point and will be dealt with in the next section. To prove the existence in the class $\X_D$, the first ingredient is a Poincar\'e-type inequality that was proved in \cite{bucgia1} with $p = 2$ and $\alpha \in [1,2]$.

\bl\label{lemma.poincare}
Let $p>1$, $\alpha\in[1,p]$ and $b,m>0$. For every $u\in\sbv(\R^d)$ that satisfies the inequality $|\{u\ne0\}|\le m$, there holds
\be\label{eq.poincare}
\int_{\R^d}|\nabla u|^p\,dx+b\int_{J_u}\left[|u^+|^p+|u^-|^p\right]\,d\cH^{d-1}
\ge\lambda_{b,\alpha}(B)\left(\int_{\R^d}|u|^\alpha\,dx\right)^{p/\alpha},
\ee
where $B$ is a ball of measure $m$ and $\lambda_{b,\alpha}(B)$ the first Robin eigenvalue. Moreover, the equality holds if and only if $u$ is the first eigenfunction associated to
$$\lambda_{b,\alpha}(B):=\min\left\{\frac{\int_B |\nabla u|^p \, dx + b \int_{\partial B} |u|^p \, d\cH^{d-1}}{\left(\int_B | u|^\alpha \, dx\right)^{p/\alpha}}\ :\ u\in W^{1,p}(B)\setminus\{0\}\right\} .
$$
\el

\brmk
This inequality can be extended to any value of $p>1$ because in \cite{bucgia1} one can work with $\sbv^{1/p}(\R^d)$ in place of $\sbv^{1/2}(\R^d)$ making minimal changes.
\ermk

We are now ready to prove the existence of a solution in the class $\X_D$ using the lower semicontinuity and compactness results obtained in Section \ref{sec:preliminaryresults}.

\begin{proof}[Proof of Theorem \ref{thm.main.1}] We divide the proof into two steps, but first we recall that the notion of convergence on $\X_D$ is the following one:
$$
u_n \xrightarrow{\X_D} u \iff \begin{cases} u_n \vee \eps \xrightarrow{\gsbv(\R^d)} u \vee \eps \\ u_n \wedge (-\eps) \xrightarrow{\gsbv(\R^d)} u \wedge (-\eps) \end{cases} \quad \text{for all $\eps > 0$}.
$$

\proofpart{1}{Coercivity of the functional $\cF$}

Let $(u_n)_{n \in\N}$ be a sequence in $\X_D$ such that $\cF(u_n) \le C$ for a suitable positive constant $C$. The function $u_n$ is not in $\sbv(\R^d)$ so we cannot apply the Poincar\'e-type inequality mentioned above directly. However, by definition, the truncated function defined as
$$v_{n,\eps} := (u_n - \eps) \vee 0 + (u_n + \eps)\wedge0$$
belongs to $\gsbv(\R^d)$ for all $\eps>0$; to replace $u_n$ by $v_{n,\eps}$ we first notice that combining \eqref{eq.f.1} and \eqref{eq.g.1} with $\cF(u)\le C$ leads to a more precise estimate, namely
$$L\int_{\R^d} |\nabla u_n|^p \, dx + \beta_1 \int_{J_{u_n}} \left[ |u_n^+|^q + |u_n^-|^q \right] \, d \cH^{d-1} \le C' - \int_{\{u_n \neq 0\}} j(x,u_n,0) \, dx,$$
where $\beta_1 := \min_{x \in \R^d} \beta_1(x) > 0$. Using assumption \eqref{eq.f.2} on the right-hand side, we immediately deduce that
\be \label{eq.coerc22}
L \int_{\R^d} |\nabla u_n|^p \, dx + \beta_1 \int_{J_{u_n}} \left[ |u_n^+|^q + |u_n^-|^q \right] \, d \cH^{d-1} \le C_1 + \int_{\R^d}f(x) |u_n|^q \, dx,
\ee
where
$$C_1:= C' + \int_{\R^d} a(x) \, dx.$$
Now $\nabla v_{n,\eps}$ coincides with $\nabla u_n$ in $\{|u_n|\ge\eps\}$ and is equal to zero otherwise so that the following inequality is satisfied:
$$
\int_{\R^d} |\nabla u_n|^p \, dx \ge \int_{\R^d} |\nabla v_{n,\eps}|^p \, dx.
$$
Furthermore, it is easy to notice that
$$
\int_{J_{u_n}} \left[ |u_n^+|^q + |u_n^-|^q \right] \, d \cH^{d-1}\ge\int_{J_{v_{n,\eps}}} \left[ |v_{n,\eps}^+|^q + |v_{n,\eps}^-|^q \right] \, d \cH^{d-1}
$$
since $J_{v_{n,\eps}}\subset J_{u_n}$ for all $\eps>0$ and $|u_n| \geq |v_{n,\eps}|$ by construction. By assumption \eqref{eq.f.2.2} $f$ is in $L^\infty$ so that, using the fact that $|v_{n,\eps}\pm \eps| \ge |u_n|$ when $u_n$ is respectively positive or negative, we obtain
\[\begin{aligned}
\int_{\R^d} f(x)|u_n|^q\,dx
&\le\|f\|_{L^\infty(\R^d)}\int_{\R^d}|u_n|^q\,dx\\
&\le\|f\|_{L^\infty(\R^d)}\left[\int_{\{u_n\ge0\}} |v_{n,\eps}+\eps|^q\,dx+\int_{\{u_n\le0\}} |v_{n,\eps}-\eps|^q\,dx\right]\\
&\le q\|f\|_{L^\infty(\R^d)}\left[\int_{\R^d} |v_{n,\eps}|^q\,dx+\eps^q |D|\right].
\end{aligned}\]
Therefore, from \eqref{eq.coerc22} we obtain a similar estimate for $v_{n,\eps}$ that reads as
\be\label{eq.a.14}
\begin{split}
L\int_{\R^d} |\nabla v_{n,\eps}|^p\,dx
&+\beta_1\int_{J_{v_{n,\eps}}}\left[|v_{n,\eps}^+|^q+|v_{n,\eps}^-|^q \right]\,d\cH^{d-1}\\
&\le C_1+q\|f\|_{L^\infty(\R^d)}\left[\int_{\R^d} |v_{n,\eps}|^q\,dx+\eps^q|D|\right].
\end{split}
\ee
The assumption $p\ge q$ comes into play here because of the Poincar\'e-type inequality and of the fact that $1+|\nabla u|^p \ge |\nabla u|^q$, which gives
\be \label{eq.ee2}
L\int_{\R^d} |\nabla u|^p\,dx+L|D|\ge L\int_{\R^d} |\nabla u|^q\,dx.
\ee
Now plug \eqref{eq.ee2} into \eqref{eq.a.14} and let $C_2 := C_1 + L|D|$ to obtain
\[\begin{aligned}
\int_{\R^d} |\nabla v_{n,\eps}|^q \, dx + \frac{\beta_1}{L} & \int_{J_{v_{n,\eps}}} \left[ |v_{n,\eps}^+|^q + |v_{n,\eps}^-|^q \right] \, d \cH^{d-1} 
\\ & \le \frac{C_2}{L} + \frac{q \|f\|_{L^\infty(\R^d)}}{L} \left[ \int_{\R^d} |v_{n,\eps}|^q \, dx + \eps^q |D| \right].
\end{aligned}\]
The Poincar\'e-type inequality \eqref{eq.poincare} - which holds for $\gsbv$ functions as well - allows us to bound the left-hand side from below so that the estimate above leads to
$$
\| v_{n,\eps} \|_{L^q(\R^d)}^q \left[\lambda_{\frac{\beta_1}{L},q}(B) - \frac{q \|f\|_{L^\infty(\R^d)}}{L}\right] \leq C_L'' + \frac{q \|f\|_{L^\infty(\R^d)}}{L} \eps^q |D|
$$
which, taking into account that $2q\|f\|_{L^\infty(\R^d)}\le L\lambda_{\beta_1/L,q}(B)$ by assumption \eqref{eq.f.2.2}, gives
$$
\| v_{n,\eps} \|_{L^q(\R^d)} \leq \tilde{C}_\eps := \tilde{C}_1 + \tilde{C}_2 \eps.
$$
This means that $v_{n,\eps}$ is uniformly bounded in $L^q(\R^d)$ as the constant on the right-hand side is independent of $n$. It follows from \eqref{eq.coerc22} that
$$
\int_{J_{u_n}} \left[ |u_n^+|^q + |u_n^-|^q \right] \, d \cH^{d-1} \leq \frac{\tilde{C}_1 + \tilde{C}_2 \eps}{\beta_1}.
$$
If we now denote by $J_{u_n}^{\geq \eps}$ the set of all jumps of $u_n$ which are bigger than (or equal to) $\eps$, namely the set
$$
J_{u_n}^{\geq \eps} = J_{u_n} \cap\left( \{ u_n^+ \leq - \eps \} \cup \{ u_n^- \geq \eps\} \right),
$$
then the estimate above can be rewritten as
$$
\eps^q \cH^{d-1} \left(J_{u_n}^{\geq \eps}\right) \leq \frac{\tilde{C}_1 + \tilde{C}_2 \eps}{2 \beta_1}
$$
since $|u_n^+|^q + |u_n^-|^q \geq 2 \eps^q$ by definition. Therefore, for $\eps$ small enough we have
$$
\cH^{d-1}\left(J_{u_n}^{\geq \eps}\right) \leq \frac{\tilde{C}_1}{2\beta_1} \eps^{-q},
$$
and, using the fact that $J_{v_{n,\eps}} = J_{u_n}^{\ge\eps}$ leads to
$$
\cH^{d-1}\left(J_{v_{n,\eps}}\right)\le\frac{\tilde{C}_1}{2\beta_1}\eps^{-q}.
$$
The right-hand side is uniformly bounded with respect to $n$, and hence we can apply the compactness Theorem \ref{compactness.sbv} (more precisely, Remark \ref{cpt.ms}) to infer that
$$
v_{n_k,\eps} \xrightarrow{k\to \infty} u_\eps \in \gsbv(\R^d)
$$
with respect to $L^1(\R^d)$ convergence. We do not achieve $\sbv$-regularity because Ambrosio's theorem requires a uniform bound on the $L^\infty$-norm. Nevertheless, a diagonal argument and pointwise convergence show that there exists some $\bar{u}\in\X_D$ such that $u_\eps=\bar{u}\vee\eps$.

\proofpart{2}{Lower semicontinuity of the functional $\cF$}
Let $(u_n)_{n \in \N} \subset \X_D$ be a sequence converging to $u$ and let $v_{n,\eps}$ be defined as in the first step. With no loss of generality we may assume $\cF(u_n)\le C$ for a suitable positive constant $C$. We first rewrite $\cF(u)=\cF_1(u)+\cF_2(u)$, where
\[\begin{aligned}
&\cF_1(u):=\int_{\R^d} \left(j(x,u,\nabla u)-j(x,u,0)\right)\,dx+\int_{J_{u_n}}\left[g(x,u^+)+g(x,u^-)\right]\,d\cH^{d-1}\\
&\cF_2(u):=\int_{\R^d} j(x,u,0)\,dx.
\end{aligned}\]
We now apply Ambrosio's lower semicontinuity Theorem \ref{lsc.sbv} to infer that the $\cF_1$ is lower semicontinuous in $\gsbv(\R^d)$, which means that
$$
\liminf_{n \to \infty} \cF_1(v_{n,\eps}) \geq \cF_1(v_\eps),
$$
where $v_\eps$ is the limit of $v_{n,\eps}$ in $\gsbv(\R^d)$ for $\eps > 0$ fixed. It remains to prove that $\cF_2$ is lower semicontinuous or, in other words, that
$$
\int_{\R^d} j(x,v_\eps,0)\,dx\le\liminf_{n\to\infty}\int_{\R^d}j(x,v_{n,\eps},0)\,dx.
$$
We first apply Fatou's lemma taking into account the estimate \eqref{eq.f.2} and deduce that
\[\begin{aligned}
\int_{\R^d} j(x,v_\eps,0)\,dx&+\int_{\R^d}-f(x)|v_\eps|^q\,dx\\
&\le\liminf_{n\to\infty}\int_{\R^d}j(x,v_{n,\eps},0)\,dx+\liminf_{n\to\infty}\int_{\R^d}-f(x)|v_{n,\eps}|^q,
\end{aligned}\]
and the conclusion follows if we are able to prove that $v_{n,\eps} \to v_\eps$ strongly in $L^q$. Using the same argument of the coercivity step, and the fact that $\cF(u_n)\le C$, we apply Lemma \ref{lemma.aux1} to $v_{n,\eps}$ to deduce the strong convergence in $L^q(\R^d)$ to $v_\eps$, concluding the proof.
\end{proof}

\brmk
The function $\bar{u}$ obtained so far is not, a priori, in $\sbv(\R^d)$ and not even in $\gsbv(\R^d)$, because the estimate
$$\cH^{d-1}(J_{u_\eps})\lesssim\eps^{-q}$$
does not provide any upper bound to $\cH^{d-1}(J_{\bar{u}})$. We will prove in next section that the $L^\infty$-norm of $u$ is bounded so that
$$\bar{u}_\eps:=(\bar{u}-\eps)\vee0+(\bar{u}+\eps)\wedge0\in\sbv(\R^d).$$
Next, under additional assumptions for which $\bar{u}$ will be a nonnegative function, we will also show that
$$\bar{u}(x)\ge\alpha>0\qquad\text{for a.e. }x\in\{x\ :\ \bar{u}(x)\ne0\},$$
which gives $\bar{u}\in\sbv(\R^d)$, hence $\bar{u}\in\mathscr{F}_D$.
\ermk

We conclude this section by giving a sufficient condition on the function $j$ such that the solution $\bar{u}$ can be taken positive. This is important in the proof of Theorem \ref{thm.main.2}.

\begin{corollary} \label{cor.pos}
Under the same assumptions of Theorem \ref{thm.main.1}, if the integrand $j$ also satisfies (j4), then the solution $\bar{u}$ is nonnegative, and therefore
$$
\bar{\Om} = \{ \bar{u} \neq 0 \} = \{ \bar{u} > 0 \}.
$$
\end{corollary}

\begin{proof}
Let $E := \{ \bar{u} < 0 \} \subset \bar{\Om}$ and define $\bar{u}_E$ as the function that coincides with $\bar{u}$ in $\R^d\setminus E$ and is equal to zero in $E$. Then (taking $L = 1$ for simplicity) we have
\[ \begin{aligned}
\cF(\bar{u}) - \cF(\bar{u}_E) & \ge \int_E \left[ j(x,\bar{u},\nabla \bar{u})-j(x,0,0) \right] \, dx 
\\ & \ge \int_E \left[|\nabla \bar{u}|^p + j(x,\bar{u},0) - j(x,0,0) \right] \, dx,
\end{aligned} \]
and the latter is nonnegative because $j(x,\bar{u},0) - j(x,0,0) \geq 0$ as $j(x,s,0)$ is nonincreasing with respect to $s$ and $\bar{u}<0$ in $E$.
\end{proof}

%%%%%%%%%%%%%%%%%%%%%%%%%%%%%%%%%%%%%%%%%%%%%%%%%%%%%%%%%%%%
\section{Proof of Theorem \ref{thm.main.2}: $\bar{\Om}$ has finite perimeter}

So far we have been able to prove that a solution to the shape optimization problem exists and has the form $\bar{\Om}=\{\bar{u}\ne0\}$. By the implication
$$\bar{u} \in L^\infty(\R^d) \implies \text{$\bar{u} \wedge (-\eps), \, \bar{u} \vee \eps \in \sbv(\R^d)$ for all $\eps > 0$},$$
our goal becomes to prove that $\bar{u}\in L^\infty(\R^d)$ and, under additional assumptions on the integrands $j$ and $g$, $\bar{u}$ is $\sbv$-regular because this gives (see Section \ref{sec.sbvregularit}) the property $\per(\bar{\Om})<\infty$. This is achieved by finding $\alpha > 0$ such that
$\bar{u}\ge\alpha$ almost everywhere in $\{\bar{u}\ne0\} $ in Theorem \ref{boundbelow}. Consequently, we deduce $\bar{u}\in\mathscr{F}_D$ and $\bar{\Om}$ has finite perimeter.

%%%%%%%%%%%%%%%%%%%%%%%%%
\subsection{Minimizers of $\cF$ are $L^\infty$-regular} 

To prove that $\bar{\Om}$ has finite perimeter, we first need to show that any minimizer $\bar{u}$ of the functional \eqref{eq.intro2} in $\mathscr{F}_D$ is bounded from above.

\bl \label{lemma.boundinf}
Let $j$, $g$ and $\bar{u}$ be as in Theorem \ref{thm.main.1} and further assume that $j$ satisfies the assumption (j4). Then there exists $M > 0$ such that $\|\bar{u}\|_\infty\le M$. In particular, $\bar{u}$ belongs to the functional space
$$
\left\{ u \: : \: u\vee\eps,\ u\wedge(-\eps)\in\sbv(\R^d)\ \forall\eps>0,\text{ $u = 0$ on $\R^d \setminus D$} \right\}.
$$
\el

The proof is based on the argument of \cite[Theorem 12]{berbuc2013}, but here Poincar\'e inequality \eqref{eq.poincare} plays a key role. To fix some notation, consider the rescaling function
$$r(M):=\frac{|D|^{1/d}}{|\Om_M|^{1/d}}\;,\qquad\hbox{where }\Om_M:=\bar{\Om}\cap\{\bar{u}>M\},$$
and the scaled set $\Om_M^\#:=r(M)\cdot\Om_M$ in such a way that $|\Om_M^\#|=|D|$.

\begin{proof}
Choose $\bar{u}_M := \bar{u} \wedge M$ as a test function (recall that $\bar{u}$ is nonnegative as a consequence of assumption (j4)) and use the minimality of $\bar{u}$ to write
$$\cF(\bar{u}) \leq \cF(\bar{u}\wedge M).$$
We obtain
\[\begin{aligned}
\int_{\Om_M}
\big(j(x,\bar{u},\nabla\bar{u})&-j(x,M,0)\big)\,dx + \int_{ J_{\bar{u}} \cap\{ \bar{u}^- > M \} } \left[ g(x,\bar{u}^+)+g(x,\bar{u}^-) \right]\,d\cH^{d-1}\\
&+\int_{J_{\bar{u}}\cap\{\bar{u}^-<M<\bar{u}^+\}} g(x,\bar{u}^+)\,d\cH^{d-1}\le0,
\end{aligned}\]
and we notice that the last addendum is positive. By assumption \eqref{eq.g.1} we get
$$
\int_{ J_{\bar{u}} \cap\Om_M } \left[g(x,\bar{u}^+)+g(x,\bar{u}^-)\right] \, d\cH^{d-1} \ge \beta_1 \int_{J_{\bar{u}}\cap\Om_M} \left[|\bar{u}^+|^q + |\bar{u}^-|^q \right]\, d\cH^{d-1}.
$$
Similarly, we can write
$$
\int_{\Om_M} \left(j(x,\bar{u},\nabla\bar{u})-j(x,M,0)\right) \, dx = \int_{\Om_M} \left(j(x,\bar{u},\nabla\bar{u})\pm j(x,\bar{u},0) - j(x,M,0)\right) \, dx
$$
and, using assumption \eqref{eq.f.1} yields
$$
\int_{\Om_M} \left(j(x,\bar{u},\nabla\bar{u})- j(x,\bar{u},0) \right) \, dx \ge L \int_{\Om_M} |\nabla \bar{u}|^p \, dx \ge L \int_{\Om_M} |\nabla \bar{u}|^q \, dx - L |\Om_M|.
$$
Putting all these inequalities together gives
\[\begin{aligned}
L\int_{\Om_M} |\nabla \bar{u}|^q\,dx&+\beta_1\int_{J_{\bar{u}}\cap\Om_M}\left[|\bar{u}^+|^q+|\bar{u}^-|^q\right]\,d\cH^{d-1}\\
& \le L|\Om_M| + \int_{\Om_M} \left(j(x,M,0)-j(x,\bar{u},0)\right) \, dx,
\end{aligned}\]
which, as a consequence of \eqref{eq.f.5}, gives
$$\int_{\Om_M} |\nabla \bar{u}|^q\,dx+\frac{\beta_1}{L}\int_{J_{\bar{u}}\cap\Om_M}\left[|\bar{u}^+|^q+|\bar{u}^-|^q \right]\,d\cH^{d-1}\le\left(1+\frac{C_j}{L}M^q\right)|\Om_M|.$$
Let $v := \max\{ \bar{u} - M, 0\}$. A simple computation shows that
\[\begin{aligned}
&\int_{\Om_M} |\nabla \bar{u}|^q\, dx = \int_{\R^d} |\nabla v|^q \, dx, \qquad J_v = J_{\bar{u}} \cap \Om_M,\\
&\int_{J_{\bar{u}}\cap\Om_M}\left[|\bar{u}^+|^q + |\bar{u}^-|^q \right]\, d\cH^{d-1} \ge \int_{J_v}\left[|v^+|^q + |v^-|^q\right]\,d\cH^{d-1}.
\end{aligned}\]
We can therefore apply Poincar\'e inequality \eqref{eq.poincare} to the function $v$ and infer that
$$
\int_{\Om_M} |\nabla \bar{u}|^q \, dx + \frac{\beta_1}{L} \int_{J_{\bar{u}}\cap\Om_M} \left[|\bar{u}^+|^q + |\bar{u}^-|^q \right]\, d\cH^{d-1} \ge \lambda_{\frac{\beta_1}{L},q}(\Om_M) \int_{\R^d} |v|^q \, dx,
$$
and the right-hand side can easily be rewritten by noticing that
$$\int_{\R^d} |v|^q \, dx = \int_{\Om_M} |\bar{u} - M|^q \, dx.$$
Now apply H\"older's inequality to estimate the latter from below, namely
$$\int_{\Om_M} |\bar{u} - M|^q \, dx \ge |\Om|^{-q/q'} \left[ \int_{\Om_M} (\bar{u} - M) \, dx \right]^q.$$
Let $f(M) := \int_{\Om_M} (\bar{u} - M) \, dx$ and put everything together to rewrite the inequality as
$$\lambda_{\beta_1/L,q}(\Om_M) f(M)^q \le \left(1 + \frac{C_j}{L} M^q \right) |\Om_M|^{q}.$$
The inequality is not precise enough because $f(M)$ and $|\Om_M|$ have the same exponent, so the idea is to now introduce the scaling $\Om_M^\#$ and write
$$\lambda_{\beta_1r(M)^{1-q}/L,q}(\Om_M^\#) f(M)^q r(M)^q \le \left(1 + \frac{C_j}{L} M^q \right) |\Om_M|^{q},$$
which, using the definition of $r(M)$, gives
$$\lambda_{\beta_1r(M)^{1-q}/L,q}(\Om_M^\#) f(M)^q r(M)^{q-1} |D|^{1/d} \le \left(1 + \frac{C_j}{L} M^q \right) |\Om_M|^{q+1/d}.$$
Since $|\Om_M| = - f'(M)$, we can rewrite the inequality as
\[\begin{aligned} 
-\frac{f'(M)}{f(M)^{qd/(qd+1)}}
&\ge\left[\left(1 + \frac{C_j}{L} M^q \right)^{-1} \lambda_{\beta_1r(M)^{1-q}/L,q}(\Om_M^\#) r(M)^{q-1}\right]^{d/(qd+1)}|D|^{1/(qd+1)}\\
&\ge\left[\left(1+\frac{C_j}{L} M^q \right)^{-1}\lambda_{\beta_1r(M)^{1-q}/L,q}(B) r(M)^{q-1} \right]^{d/(qd+1)}|D|^{1/(qd+1)},
\end{aligned} \]
where $B$ is a ball of volume $|D|$. Indeed, the eigenvalue $\lambda_{b,q}(\cdot)$ is minimized (at fixed volume) by the ball, so
$$
\lambda_{b,q}(\Om_M) \ge \lambda_{b,q}(B_M),\quad \text{with $|B_M| = |\Om_M|$}.
$$
On the other hand, given a domain $C$ and a positive parameter $t$, it is
$$
\lambda_{b,q}(tC) = t^{-2} \lambda_{b,q}(C),
$$
and therefore $\lambda_{b,q}(B_r)$ is monotone decreasing with respect to the radius $r$, hence with respect to the volume. In particular, taking into account that $\Om_M\subset D$, we have
$$\lambda_{\beta_1r(M)^{1-q}/L,q}(\Om_M^\#) \ge \lambda_{\beta_1r(M)^{1-q}/L,q}(B).$$
Integrating in $M$ between $0$ and $T < \|\bar{u}\|_\infty$ and taking into account that $f$ is a positive function yields
$$
f(0)^{1/(qd+1)} \ge c_{|D|,q} \int_0^T \left[ \left(1 + \frac{C_j}{L} M^q \right)^{-1} \lambda_{\beta_1r(M)^{1-q}/L,q}(B) r(M)^{q-1} \right]^{d/(qd+1)}\,dM.
$$
The left-hand side is bounded because $f(0)$ is bounded from above by $\|\bar{u}\|_{L^1(\R^d)}$; as for the right-hand side, we apply \cite[Lemma 13]{berbuc2013} to infer that
$$
\lim_{M\to\|\bar{u}\|_\infty}\left(1+\frac{C_j}{L} M^q\right)^{-1}\lambda_{\beta_1r(M)^{1-q}/L,q}(B) r(M)^{q-1}=\left(1+\frac{C_j}{L}\|\bar{u}\|_\infty^q \right)^{-1} \frac{\beta_1}{L} d.
$$
This yields $\| \bar{u} \|_\infty < \infty$ as a consequence of the fact that $qd/(qd+1)\in(0,1)$, concluding the proof.
\end{proof}

\brmk
Thanks to the $L^\infty$ estimate above, we deduce that $\bar{u}$ given in the proof of Theorem \ref{thm.main.1} belongs to
$$\big\{ u \ :\ u\vee\eps,\ u\wedge(-\eps)\in\sbv(\R^d)\ \forall\eps>0,\text{ $u = 0$ on $\R^d\setminus D$}\big\}.$$
\ermk

%%%%%%%%%%%%%%%%%%%%%%%%%
\subsection{Bound from below}\label{sec.sbvregularit}

To prove that $\bar{u}$ is bounded away from zero, we follow the approach of \cite[Theorem 3.5]{bugi16}. 

\bd[Supersolution]
We say that $w \in \mathscr{F}_D \cap \{u \geq 0\}$ is a {\em supersolution} for the functional
$$
\cL(u) := \int_{\R^d} \left[ j(x,u,\nabla u) - j(x,0,0)\right] \, dx + \int_{J_u} \left[ g(x,u^+) + g(x,u^-) \right] \, d\cH^{d-1}
$$
if for every $v \in \mathscr{F}_D \cap \{ u \geq 0\}$ with $0\leq w \leq v$ we have
$$
\cL(w) \leq \cL(v).
$$
\ed

\bl
Let $j$ and $g$ be as in Theorem \ref{thm.main.2}, and let $\bar{u}$ be the solution given in Theorem \ref{thm.main.1}. Then the positive and negative parts $\bar{u}_{\ge0}:=\bar{u}\vee 0$ and $\bar{u}_{\le0}:=(-\bar{u})\vee0$ are supersolutions for $\cL$.
\el

\begin{proof}
Suppose that $\bar{u}_{\ge0}$ is not a supersolution of $\cL$ and let $v\in\mathscr{F}_D$ be a function with $0\le\bar{u}_{\ge0}\le v$ satisfying the strict inequality $\cL(v)<\cL(\bar{u}_{\ge0})$. To find a contradiction we define the function
\[
\bar{v}(x) := \begin{cases} \bar{u}(x) & \text{if $\bar{u}(x) \leq 0$}, \\ v(x) & \text{if $\bar{u}(x) > 0$},
\end{cases}
\]
and we obtain
\[\begin{aligned}
\cF(\bar{u}) &- \cF(\bar{v}) \geq \int_{ \{\bar{u} > 0 \} } \left[ j(x,\bar{u},\nabla \bar{u})-j(x,v,\nabla v) \right] \, dx
\\[.6em] & + \int_{ \{ \bar{u} > 0 \} \cap J_{\bar{u}} } \left[ g(x,\bar{u}^+) + g(x,\bar{u}^-) \right]\,d\cH^{d-1}-\int_{ \{ \bar{u} > 0 \} \cap J_{v} } \left[ g(x,v^+) + g(x,v^-) \right] \, d\cH^{d-1}
\end{aligned} \]
because $\bar{u}=\bar{v}$ on $\{\bar{u}< 0\}$. Finally, we use the inequality $\cL(v) < \cL(\bar{u}_{\geq 0})$ to infer that
$$
\cF(\bar{u}) - \cF(\bar{v}) > 0,
$$
which is in contradiction with the fact that $\bar{u}$ is a minimizer for $\cF$.

A similar proof can be done for the function $\bar{u}_{\le0}$ too.
\end{proof}

\brmk
Despite the fact that both $\bar{u}_{\geq 0}$ and $-\bar{u}_{\leq 0}$ are supersolutions for $\cF$, we still need to assume that (j4) holds so that
$$
\bar{u} > 0 \implies \bar{u} = \bar{u}_{\geq 0} \quad \text{(see Corollary \ref{cor.pos})}.
$$
In particular, in our case $\bar{u}$ is itself a supersolution for the functional $\cL$.
\ermk

\bt \label{boundbelow}
Let $j$ and $g$ be as in Theorem \ref{thm.main.2} and let $u$ be a supersolution in the sense above. Then there exists a positive $\alpha$ such that
\be\label{eq.boundbelow}
u \geq \alpha \quad \text{a.e. on $\{u>0\}$.}
\ee
\et

\begin{proof}
Let $\eps > 0$ be such that $u_\eps := \max\{u,\eps\} \in \sbv(\R^d)$. Then $u_\eps \in \mathscr{F}_D$ and, by comparison with $u$, we find that
\[ \begin{aligned}
\int_{\R^d} &\left[ j(x,u,\nabla u) - j(x,0,0)\right] \, dx + \int_{J_u} \left[ g(x,u^+) + g(x,u^-) \right] \, d\cH^{d-1} 
\\ & \leq \int_{\R^d} \left[ j(x,u_\eps,\nabla u_\eps) - j(x,0,0)\right] \, dx + \int_{J_{u_\eps}} \left[ g(x,u_\eps^+) + g(x,u_\eps^-) \right] \, d\cH^{d-1}.
\end{aligned} \]
Using (j2), (g2) and (g4) we infer that
\[\begin{aligned}
\int_{\{u \leq \eps\}}& \left[|\nabla u|^p + j(x,u,0) - j(x,\eps,0)\right] \, dx + \beta_1 \int_{J_u \cap \{u^- < u^+ \leq \eps\}} \left[ |u^+|^q + |u^-|^q \right] \, d\cH^{d-1}\\
&\le\beta_2\eps^q \cH^{d-1} \left(\partial^e \{ u > \eps \} \setminus J_u \right),
\end{aligned}\]
because the jump part on $\{u^-\ge\eps \} \cap J_u$ is the same while 
$$\int_{\{u^-\le\eps<u^+\}\cap J_u}\left[g(x,u^+)+g(x,u^-)\right]\,d\cH^{d-1}$$
only appears on the left-hand side of the inequality so it can be estimated from below by $0$. Finally, assumption (j4) gives that $j(x,\cdot,0)$ is decreasing, so that
$$\int_{\{u\le\eps\}}\left[j(x,u,0)-j(x,\eps,0)\right]\,dx\ge0$$
and this term can be removed from the estimate above. It follows that
$$
L\int_{\{u\le\eps\}} |\nabla u|^p\,dx+\beta_1\int_{\{u^-<u^+\le\eps\}\cap J_u}\!\!\!\left[|u^+|^q+|u^-|^q\right]\,d\cH^{d-1}\le\beta_2 \eps^q \cH^{d-1} \left(\partial^e\{ u > \eps \} \setminus J_u \right).
$$
This means that for almost every $\delta$ with $0 < \delta < \eps$ we have
$$
L\int_{\{u<\eps\}}|\nabla u|^p\,dx + \beta_1\delta^q \cH^{d-1} \left( \partial^e \{ \delta < u < \eps \} \cap J_u\right) \leq \beta_2 \eps^q \cH^{d-1} \left(\partial^e\{u>\eps\} \setminus J_u \right),
$$
in such a way that, setting
\begin{gather*}
E(\eps) := \int_{\{ u \leq \eps \}} |\nabla u|^p\,dx,\qquad
\gamma(\delta,\eps) := \cH^{d-1} \left( \partial^e \{ \delta < u < \eps \} \cap J_u\right), \\
h(\eps) := \cH^{d-1} \left(\partial^e\{u>\eps\}\setminus J_u \right),
\end{gather*}
we can rewrite the inequality as
\be \label{eq.3.14}
E(\eps) + \frac{\beta_1}{L} \delta^q \gamma(\delta,\eps) \leq \frac{\beta_2}{L} \eps^q h(\eps).
\ee
At this point one can simply adapt the proof of \cite[Theorem 3.5]{bugi16}, but there are a few differences due to the fact that
$$
L E(\eps) \leq \beta_2 \eps^q h(\eps) \implies \left[ \int_{\{u<\eps\}}|\nabla u|^p\,dx \right]^{1/p}\le\left(\frac{\beta_2}{L}\right)^{1/p} \eps^{q/p} h(\eps)^{1/p}.
$$
We will now go over the proof presented in \cite{bugi16}, pointing out the main changes we need to make and why \eqref{eq.f.3} plays a key role in our case.

\proofpart{0}{Setting of the problem.}
For $\eta>0$ set
$$\eps_i:=\frac56\eta+\frac{2^{-i}}{6}\eta\qquad\text{and}\qquad\delta_i=\frac23\eta-\frac{2^{-i}}{6}\eta$$
so that $\eps_i\to\eps_\infty:=\frac56\eta$ and $\delta_i\to\delta_\infty:=\frac23\eta$ as $i \to\infty$. If we define
$$\Om(\delta,\eps):=\{ \delta < u < \eps \},$$
then we only need to show that there exists $\eta_0 > 0$ such that
\be \label{eq.3.17}
\left| \Om \left( \delta_\infty,\eps_\infty \right) \right| \int_{\delta_\infty}^{\eps_\infty} h(s)\, ds = 0 \qquad \forall \eta < \eta_0.
\ee
Indeed, the isoperimetric inequality applied to $\Om(\delta,\eps)$ gives us
$$|\Om(\delta,\eps)|^{(d-1)/d} \leq C_d\left( h(\eps) + h(\delta) + \gamma(\delta,\eps) \right),$$
and applying \eqref{eq.3.14} with $\eta/2<\delta<\eps<\eta$ yields
\be\label{formnew}
|\Om(\delta,\eps)|^{(d-1)/d} \leq \underbrace{\frac{\beta_2}{\beta_1} C_d}_{=: C_{d,\beta}} \left( 1+2^q \right) \left[h(\eps)+h(\delta)\right].
\ee
This together with \eqref{eq.3.17} is enough to infer that $|\Om(\delta_\infty,\eps_\infty)|=0$,
which means that $u$ must be at least $\frac{5}{6}\eta_0>0$ almost everywhere on its support, concluding the proof of \eqref{eq.boundbelow}.

\proofpart{1}{The main inequalities.}
Let us set for $i\in\N$
$$
a_i := \int_{\delta_i}^{\eps_i} h(s)\,ds \qquad\text{and}\qquad b_i := |\Om(\delta_i,\eps_i)|.
$$
We claim that there are positive constants $c_1,c_2$ such that
$$a_i\le c_1\frac{2^i}{\eta^{1-q/p}} a_{i-1} b_{i-1}^{1/(dp')}
\qquad\text{and}\qquad
b_i\le c_2 \left(\frac{2^i}{\eta}\right)^{d/(d-1)} a_{i-1}^{d/(d-1)},$$
where $p'=p/(p-1)$ is the conjugate exponent of $p$. The estimate of $b_i$ is obtained exactly as in \cite[Theorem 3.5]{bugi16}. As for the $a_i$, the main difference ($\eta^{q/p-1}$ in place of $\eta^0$) lies in the fact that the coarea formula gives
\[\begin{aligned}
\int_\delta^\eps h(s)\,ds
&=\int_{\Om(\delta,\eps)} |\nabla u| \, dx \le|\Om(\delta,\eps)|^{1/p'}\|\nabla u\|_{L^p(\Om(\delta,\eps))}\\
&\le|\Om(\delta,\eps)|^{1/p'}\left(\frac{\beta_2}{L}\right)^{1/p}\eps^{q/p} h(\eps)^{1/p}\\
&\le|\Om(\delta,\eps)|^{1/(dp')}\left[C_{d,\beta}(1+2^q)\right]^{1/p'}\left(\frac{\beta_2}{L}\right)^{1/p}\left[h(\eps)+h(\delta)\right]^{1/p'}\eps^{q/p} h(\eps)^{1/p}\\
&\le|\Om(\delta,\eps)|^{1/(dp')}\left[C_{d,\beta}(1+2^q)\right]^{1/p'}\left(\frac{\beta_2}{L}\right)^{1/p}\eps^{q/p}\left[h(\eps)+h(\delta)\right],
\end{aligned}\]
where we have used \eqref{formnew}. Integrating both sides of the inequality with respect to $\epsilon$ on $[\epsilon_i,\epsilon_{i-1}]$ and to $\delta$ on $[\delta_{i-1},\delta]$ leads to
$$ a_i \left[ \frac{1}{6} 2^{-i} \eta \right]^2 \le\left[C_{d,\beta}(1+2^q)\right]^{1/p'}\left(\frac{\beta_2}{L}\right)^{1/p} |\Om(\delta_{i-1},\eps_{i-1})|^{1/(dp')}\eps^{q/p} \left[ \frac{1}{6} 2^{-i} \eta \right] a_{i-1},$$
and this concludes the proof of the claim.

\proofpart{2}{Combining the main inequalities.}
We claim that we can find $\alpha>0$ such that $U_i := a_i^\alpha b_i$ satisfies the inequality
$$U_i\le\frac{\tilde{c}}{\eta^{ d/(d-1) +\alpha(1-q/p) } } A^i U_{i-1}^\vartheta,$$
where $c,A > 0$ and $\vartheta > 1$. This is once again obtained as in \cite[Theorem 3.5]{bugi16} by taking $\alpha$ and $\vartheta$ solutions of the system
$$\begin{cases}
\displaystyle\alpha + \frac{d}{d-1} = \vartheta \alpha\\
\alpha=\vartheta dp',
\end{cases}$$
which gives
\be\label{eq.alpha}
\alpha=\frac{dp'}{2}\left(1+\sqrt{1+\frac{4}{(d-1)p'}}\right),
\ee
and consequently
$$\vartheta=\frac{\alpha}{dp'}=\frac{1}{2}\left(1+\sqrt{1+\frac{4}{(d-1)p'}}\right)>1.$$

\proofpart{3}{Decay for $E(\eps)$.}
We now show that there exist $\eps_0,c_0>0$ such that
\be\label{eq.decayest}
E(\eps)\le c_0 \eps^{p(q-1)/(p-1)}\qquad \forall\eps\le\eps_0.
\ee
Indeed, using the inequality $E(\eps) \leq \frac{\beta_2}{L} \eps^q h(\eps)$ and the coarea formula yields
\[ \begin{aligned}
\eps E(\eps) & \le \int_\eps^{2 \eps} E(s) \, ds \leq \frac{\beta_2}{L} 2^q \eps^q \int_\eps^{2\eps} h(s) \, ds
\\ & = \frac{\beta_2}{L} 2^q \eps^q \int_{\Om(\eps,2\eps)} |\nabla u| \, dx \leq \frac{\beta_2}{L} 2^q \eps^q |\Om(\eps,2\eps)|^{1/p'} E(2\eps)^{1/p}.
\end{aligned} \]
Taking $\eps_0$ so small that $(\beta_2/L) 2^q |\Om(0,2\eps_0)|^{1/p'}\le1$ gives
$$E(\eps)\le\eps^{q-1} E(2\eps)^{1/p}\qquad\text{for all $\eps\le\eps_0$},$$
and now an iterative argument (see, e.g., \cite[Lemma 3.6]{bugi16}) proves \eqref{eq.decayest}.

\proofpart{4}{Conclusion.}
 We claim that we can find $\eta$ such that
\be\label{eq.3.24}
U_0\le\tilde{c}^{-1/(\vartheta-1)} A^{-1/(\vartheta-1)^2}\eta^{(q-1)/(p-1)}\alpha,
\ee
where $\alpha$ is given in {\it Step 2}. To verify \eqref{eq.3.24}, we notice that
\[\begin{aligned}
U_0 & = \left| \Om \left( \frac{\eta}{2},\eta \right) \right| \left[ \int_{\eta/2}^\eta h(s) \, ds \right]^\alpha = \left| \Om \left( \frac{\eta}{2},\eta \right) \right| \left[ \int_{\Om(\frac{\eta}{2},\eta)} |\nabla u| \, dx \right]^\alpha
\\ & \le\left|\Om\left( \frac{\eta}{2},\eta \right) \right| \left[ E(\eta)^{1/p}\left|\Om \left( \frac{\eta}{2},\eta \right) \right|^{1/p'}\right]^\alpha
\\ & = \left| \Om \left( \frac{\eta}{2},\eta \right) \right|^{1+\alpha/p'} E(\eta)^{\alpha/p}
\end{aligned}\]
so that, thanks to the decay estimate \eqref{eq.decayest} obtained above, we conclude that
$$U_0\le c_0\left|\Om\left( \frac{\eta}{2},\eta \right) \right|^{1+\alpha/p'} \eta^{(q-1)/(p-1)}\alpha.$$
This means that \eqref{eq.3.24} is achieved if $\eta$ is so small that
$$c_0\left|\Om\left(\frac{\eta}{2},\eta \right) \right|^{1+\alpha/p'}\le\tilde{c}^{-1/(\vartheta-1)} A^{-1/(\vartheta-1)^2},$$
so all it remains is to use induction on $i$ and finally prove that
$$\lim_{i \to \infty} U_i = 0.$$
A simple computation shows that
\[\begin{aligned}
U_n & \leq \tilde{c}^{\frac{\vartheta^n-1}{\vartheta-1}} \eta^{ -\frac{\vartheta^n-1}{\vartheta-1} \left[ \frac{d}{d-1} + \alpha \left(1-\frac{q}{p}\right) \right]} A^{ \frac{\vartheta(\vartheta^n-1) - n \vartheta +n}{(\vartheta - 1)^2}} (U_0)^{\vartheta^n}
\\ & \leq \tilde{c}^{-\frac{1}{\vartheta - 1}} A^{\frac{\vartheta^n}{\vartheta-1}} \eta^{ \left[ \alpha \frac{q-1}{p-1} - \frac{d}{d-1} - \alpha \left(1-\frac{q}{p}\right) \right] \vartheta^n }
\\ & \leq \tilde{c}_1 \left[ A^{\frac{1}{\vartheta-1}} \eta^{\alpha \left(\frac{q-1}{p-1} + \frac{q}{p} - 1 \right)-\frac{d}{d-1}}\right]^{\vartheta^n}
\end{aligned}\]
so $U_n \xrightarrow{n\to \infty} 0$ if we can prove that the quantity inside the parenthesis is strictly smaller than $1$ for $\eta$ sufficiently small. More precisely, we require
$$\alpha\left(\frac{q-1}{p-1}+\frac{q}{p}-1\right)-\frac{d}{d-1}>0,$$
and this is exactly condition \eqref{eq.f.3} since $\alpha$ is given by formula \eqref{eq.alpha}. In particular, thanks to assumption (j4) we conclude the proof.
\end{proof}

\brmk
The assumption \eqref{eq.f.3} seems rather restrictive, but looking at numerical simulations suggests that
$$
\alpha \left(\frac{q-1}{p-1} + \frac{q}{p} - 1 \right) - \frac{d}{d-1} > 0
$$
holds in a significant portion of the plane. Indeed, in Figure \ref{fig.1} we take $d = 2$ and introduce the function
$$q(p):=\frac{p}{2p-1}\left[p+\frac{(p-1)^2}{(d-1)p}\frac{2}{1+\sqrt{1+\frac{4(p-1)}{(d-1)p}}}\right]$$
in such a way that $r_2(p)$ is exactly the lower bound in \eqref{eq.f.3}. The range of admissible values of $p$ and $q$ in which \eqref{eq.f.3} holds is given by the portion between the two lines.
\ermk

\begin{figure}[!h]
\centering
\includegraphics[scale=.8]{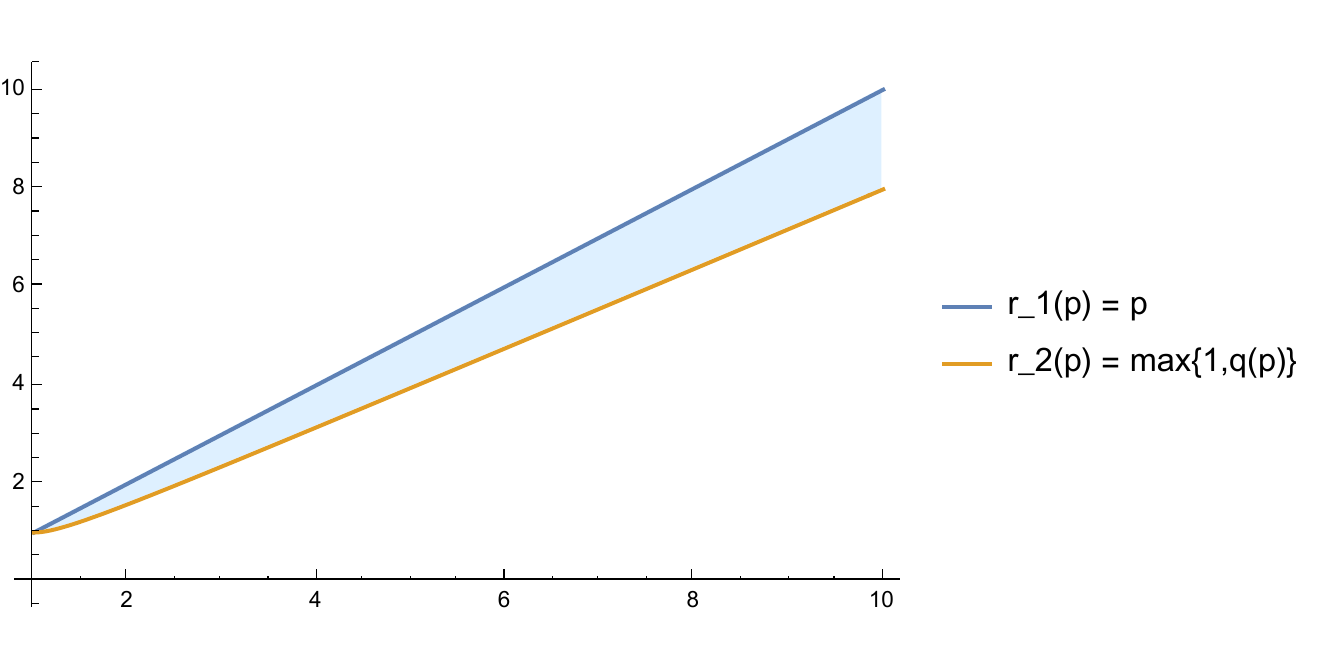}
\caption{The range of admissible values $q$ for $d=2$.}
\label{fig.1}
\end{figure}

We are now in a position to conclude the proof that $\bar{\Om} := \{ \bar{u} > 0\}$ has finite perimeter.

\begin{proof}[Proof of Theorem \ref{thm.main.2}]
If $\bar{u}$ is the solution given in Theorem \ref{thm.main.1}, then \eqref{eq.boundbelow} gives $\bar{u} = \bar{u} \vee \alpha$ for some $\alpha>0$. However, we know that the function $\bar{u} \vee \eps$ belongs to $\sbv(\R^d)$ and therefore using the bound from below and the chain rule in $\sbv$ we get
$$
\bar{u}\in\sbv(\R^d) \implies \per(\bar{\Om}) < \infty.
$$
Indeed, by definition, we have
$$
\chi_{\bar{\Om}} = \begin{cases} 0 & \text{if $\alpha^{-1} \bar{u} = 0$}, \\ 1 & \text{if $\alpha^{-1} \bar{u} \geq 1$},\end{cases}
$$
which means that $\chi_{\bar{\Om}}$ and $\alpha^{-1} \bar{u}$ are supported in the same set, namely $\bar{\Om}$. It follows that the perimeter can be estimated by
$$
\per(\bar{\Om}) = \sup_{\|\phi\|_\infty \leq 1} \int_{\R^d} \divs(\phi) \chi_{\bar{\Om}} \, dx \leq \alpha^{-1} \|\bar{u}\|_{\bv(\R^d)},
$$
and this last quantity is finite because $\bar{u} \in \sbv(\R^d)\subset\bv(\R^d)$ and $\alpha > 0$.
\end{proof}

%%%%%%%%%%%%%%%%%%%%%%%%%%%%%%%%%%%%%%%%%%%%%%%%%%%%%%%%%%%%
\section{Proof of Theorem \ref{thm.main.3}: the set $\bar{\Om}$ is open}

%%%%%%%%%%%%%%%%%%%%%%%%%
\subsection{Essential closedness of the jump set}

The goal of this section is to exploit the results obtained in \cite[Section 4]{bugi16} and adapt them to deal with our functional, taking into account that our model function is
$$j(x,u,\nabla u) = |\nabla u|^p - f(x) u + 1,$$
which leads to several issues related to the linear term $-f(x)u$. From now on, we shall always assume that $j$ satisfies the assumption (j5), which asserts that
$$j(x,u,\nabla u) - j(x,u,0) = L |\nabla u|^p.$$
Consider the associated Mumford-Shah functional
$$MS(u) := L \int_{\R^d} |\nabla u|^p \, dx + \cH^{d-1}(J_u)$$
in which the linear part does not appear. We start by recalling the notion of almost-quasi minimality for the Mumford-Shah functional:

\bd
Let $u \in \sbv_{\mathrm{loc}}^p(\R^d)$ be a function such that $u = 0$ in $\R^d \setminus D$. We say that $u$ is an almost-quasi minimizer for the functional $MS(u)$ with Dirichlet boundary conditions if there are $\Lambda \geq 1$, $\alpha > 0$ and $c_\alpha > 0$ such that
\[\begin{aligned}
\int_{B_\rho(x_0)} L |\nabla u|^p \, dx & + \cH^{d-1} \left(J_u \cap \bar{B}_\rho(x_0) \right)
\\ & \le \int_{B_\rho(x_0)} L |\nabla v|^p \, dx + \Lambda \cH^{d-1}\left(J_v \cap \bar{B}_\rho(x_0) \right) + c_\alpha \rho^{d-1 + \alpha}
\end{aligned}\]
for all $B_\rho(x_0) \subset \subset D$ and for every $v \in \sbv_{\mathrm{loc}}^p(\R^d)$, $v = 0$ in $\R^d \setminus D$ and
$$\{v\ne u\}\subset B_\rho(x_0).$$
\ed

The following result was proved in \cite[Theorem 2.3]{bugi16} with a function $f(x,\nabla u)$ in place of $j(x,u,\nabla u) - j(x,u,0)$ under suitable assumptions which are satisfied with
$$f(x,\nabla u) = C |\nabla u|^p.$$

\bt
Let $u \in \sbv_{\mathrm{loc}}^p(D)$ be an almost-quasi minimizer of $MS(u)$ with Dirichlet boundary conditions according to the definition above. Then
$$\cH^{d-1}\left((\bar{J}_u\setminus J_u)\cap D\right)=0,$$
which means that the jump set of $u$ is essentially closed in $D$.
\et

\bpr \label{prop.1aus}
Suppose that $j$ satisfies (j1)--(j5) and $g$ satisfies (g1)--(g4). Let $\bar{u} \in \sbv \cap L^\infty(\R^d)$ be the minimizer of $\cF$ given in Theorem \ref{thm.main.1}. Then
$$
\bar{u} > \alpha \quad \text{a.e. on $\{ \bar{u} > 0 \}$}
$$
for some $\alpha>0$, and the function $(2\beta_1)^{1/q} \alpha \bar{u}$ is an almost-quasi minimizer of the Mumford-Shah functional
$$
MS(u) := L \int_{\R^d} |\nabla u|^p \, dx + \cH^{d-1}(J_u)
$$
with Dirichlet boundary conditions on $D$.
\epr

\begin{proof}
The existence of the constant $\alpha$ is given in Theorem \ref{thm.main.3}, so let $B_\rho(x_0) \subset D$ and take any $v \in \sbv_{\mathrm{loc}}(D)$ satisfying $\{v\ne\bar{u}\}\subset B_\rho(x_0)$. Without loss of generality we can replace $v$ with $w := (v \wedge M) \vee 0$, where $M \geq \|\bar{u}\|_\infty$ is the constant given in Lemma \ref{lemma.boundinf}. Comparing $\bar{u}$ and $w$ we get
\[
\cF(\bar{u}) \leq \cF(w),
\]
which immediately translates to
\[ \begin{aligned}
\int_{\R^d} L |\nabla& \bar{u}|^p \, dx + \int_{\R^d} j(x,\bar{u},0) \, dx + \int_{J_{\bar{u}}} \left[ g(x,\bar{u}^+) + g(x,\bar{u}^-) \right] \, d\cH^{d-1}
\\ & \le \int_{\R^d} L |\nabla w|^p \, dx + \int_{\R^d} j(x,w,0) \, dx + \int_{J_w} \left[ g(x,w^+) + g(x,w^-) \right] \, d\cH^{d-1}.
\end{aligned}\]
Since $\bar{u}$ and $w$ coincide outside of $B_\rho(x_0)$, we find that
\[ \begin{aligned}
\int_{B_\rho(x_0)} L |\nabla & \bar{u}|^p \, dx + \int_{B_\rho(x_0)} \left[ j(x,\bar{u},0) - j(x,w,0) \right] \, dx + 2 \beta_1 \alpha^q \cH^{d-1}\left( J_{\bar{u}} \cap \bar{B}_\rho(x_0) \right)
\\ & \le \int_{\R^d} L |\nabla w|^p \, dx + 2 \beta_2 M^q \cH^{d-1} \left( J_w \cap \bar{B}_\rho(x_0) \right) + \gamma \omega_d \rho^d.
\end{aligned}\]
We now apply assumption \eqref{eq.f.5} to infer that
$$
\int_{B_\rho(x_0)} \left[ j(x,\bar{u},0) - j(x,w,0) \right] \, dx \ge - C_j |B_\rho(x_0)| \| w \|_{L^q(\R^d)}^q \ge - C_j' \rho^d,
$$
and since we have
\[ \begin{aligned}
\int_{B_\rho(x_0)} L |\nabla
&\bar{u}|^p\,dx+2 \beta_1 \alpha^q \cH^{d-1}\left( J_{\bar{u}} \cap \bar{B}_\rho(x_0) \right)\\
&\le\int_{\R^d} L |\nabla w|^p\,dx+2\beta_2 M^q \cH^{d-1} \left( J_w \cap \bar{B}_\rho(x_0) \right) + (\gamma \omega_d + C_j')\rho^d,
\end{aligned}\]
this leads to the conclusion.
\end{proof}

%%%%%%%%%%%%%%%%%%%%%%%%%
\subsection{Proof of the main result}

We now use Proposition \ref{prop.1aus} to prove that $\bar{\Om} = \{ \bar{u} \neq 0\}$ is open and thus belongs to the class $\cA(D)$. As a consequence, we also show that
$$\inf\big\{\cJ(\Om)\ :\ \Om\in\cA(D),\text{ $\Om$ Lipschitz}\big\}=\cJ(\bar{\Om}).$$

\begin{proof}[Proof of Theorem \ref{thm.main.3}]
Let $\bar{u}$ be the minimizer given in Proposition \ref{prop.1aus}. Then
\begin{equation}\label{eq.mm1}
\cH^{d-1}\big((\bar{J}_{\bar{u}}\setminus J_{\bar{u}}) \cap D\big) = 0,
\end{equation}
which means that $J_{\bar{u}}$ is essentially closed in $D$. Also, the lower bound \eqref{eq.boundbelow} gives $\cH^{d-1}(J_{\bar{u}}) < \infty$. If $\bar{\Om}$ is the connected component of $D\setminus J_{\bar{u}}$ on which $\bar{u}$ does not vanish (we have just one component by minimality), then $\partial\bar{\Om}\subset\left(\overline{J_{\bar{u}}}\cap D\right)\cup\partial D$ implies $\bar{\Om}$ open. Moreover, we have
$$
\int_{\partial \bar{\Om} \setminus J_{\bar{u}}} \left[ g(x,\bar{u}^+) + g(x,\bar{u}^-) \right] \, \cH^{d-1} = 0
$$
since \eqref{eq.mm1} implies
$$
\cH^{d-1}\big((\partial \bar{\Om} \setminus J_{\bar{u}}) \cap D\big)=0,
$$
while $\bar{u}^{\pm} = 0$ almost everywhere on the portion on the boundary of $D$, because the latter is Lipschitz and hence $\bar{u}^-= 0$ almost everywhere on $\partial D$. To prove that
$$
\inf\big\{\cJ(\Om)\ :\ \Om\in\cA(D),\text{ $\Om$ Lipschitz}\big\}=\cJ(\bar{\Om})
$$
we use Proposition \ref{prop.2aus} to find for each $\eps > 0$ a function $w \in \mathscr{F}_D$ with $J_w \subset D$ such that
$$\cF(w) < \cF(\bar{u}) + \eps. $$
We now follow the approach of \cite{bugi16}. By \cite[Theorem 3.1]{cortoa}, we can find a sequence $w_k \in \mathscr{F}_D$ with $J_{w_k}$ essentially closed and polyhedral, $w_k \in W^{1,p}(D\setminus J_{w_k})$ and such that
\[ \begin{aligned}
& w_k \xrightarrow{k \to \infty} w \quad \text{strongly in $L^p(\R^d)$},
\\ & \nabla w_k \xrightarrow{k \to \infty} \nabla w \quad \text{strongly in $L^p(\R^d)$},
\\ & \cF(w_k) \xrightarrow{k\to \infty} \cF(w).
\end{aligned} \]
The set $D \setminus \bar{J}_{w_k}$ is open, but the boundary is only Lipschitz outside of a $\cH^{d-1}$-measure zero (because $J_{w_k}$ is essentially closed) subset $A \subseteq \bar{J}_{w_k} \setminus J_{w_k}$. We can cover $A$ by arbitrarily small ``holes'' with polyhedral boundary $B_1,\ldots,B_N$ in such a way that
$$H_k:=J_{w_k}\cup\bigg(\bigcup_{j = 1}^N B_j \bigg)$$
has polyhedral boundary (thus Lipschitz). Now the set $\Om_k:=D\setminus\bar{H}_k$ belongs to $\cA(D)$ and has a Lipschitz boundary, and therefore the restriction of $w_k$ to $\Om_k$ is a competitor for the functional $\cJ(\cdot)$. We can consider holes so small that
$$\cJ(\Om_k) \le \int_{\Om_k} j(x,w_k,\nabla w_k) \, dx + \int_{\partial \Om_k} g(x,w_k) \, d\cH^{d-1} \le \cF(w_k) + \eps \le \cJ(\bar{\Om}) + \eps$$
holds for $k$ sufficiently large so, by taking a sequence $\eps_n\to0$ and a corresponding sequence of $k_n$ for which the inequality above holds, we obtain the thesis.
\end{proof}

\bpr\label{prop.2aus}
Let $v \in \mathscr{F}_D \cap L^\infty(\R^d)$ with $\cH^{d-1}(J_v)<\infty$. For all $\eps > 0$ there is $w \in \mathscr{F}_D\cap L^\infty(\R^d)$ such that
$$
J_w \subset D,\qquad\cH^{d-1}(J_w)<\infty,\qquad\cF(w) \le \cF(v) + \eps.
$$
\epr

The proof follows the same argument as in \cite[Proposition 3.12]{bugi16}, with $B = \varnothing$, but there is an important difference which is the definition of $v^\xi$. More precisely, we set
\[ 
v_i^\xi(y) := \begin{cases} v_i(y',y_d+\xi) & \text{if $y_d < f_i(y')-\xi$}, \\ \|v\|_\infty \psi_i(y',f_i(y')) & \text{if $y_d \ge f_i(y')-\xi$}, \end{cases}
\]
where $\psi_i$ is the partition of unity introduced in the reference paper. The rest of the proof follows in the same way because the term
$$\int_{\R^d}\left[ j(x,v^\xi,0)-j(x,v,0)\right]\,dx$$
can be easily estimated by a constant when $\xi > 0$ is sufficiently small, taking into account that the support of $v_i^\xi$ is as close as we want to the one of $v_i$ by definition.

%%%%%%%%%%%%%%%%%%%%%%%%%%%%%%%%%%%%%%%%%%%%%%%%%%%%%%%%%%%%
\section{Further comments and open problems}\label{sec:openpbs}

In this section we raise some question that look, in our opinion, very interesting and that could help to better understand the shape optimization problems with Robin conditions at the free boundary.

\medskip

{\bf Question 1. }We obtained under very mild assumptions the existence of an optimal domain $\bar{\Om}$ in the class of measurable subsets of $D$. Under some slightly stronger assumptions we showed that $\bar{\Om}$ has a finite perimeter and that $\bar{\Om}$ is actually an open set. It would be very interesting to investigate about further regularity properties of $\bar{\Om}$. For instance, according to the results in \cite{buccou}, the boundary of $\bar{\Om}$ cannot have too sharp cuspids, and also $\bar{\Om}$ does not have too many internal fractures. The question if boundary cusps or internal fractures may actually occur for an optimal domain $\bar{\Om}$ is still open.

\medskip

{\bf Question 2. }The investigation on the higher regularity of optimal domains is also interesting. Taking the model case
\be\label{modelcase}
j(x,s,z)=|z|^p-f(x)s+1,\qquad g(x,s)=|s|^p\qquad\text{with }p>1
\ee
is it possible to obtain $C^{1,\alpha}$ regularity of the free boundary assuming $f$ nonnegative and bounded?

\medskip

{\bf Question 3. }In our model case \eqref{modelcase} a key assumption in order to show the existence of optimal sets $\bar{\Om}$ that are open and with finite perimeter is the nonnegativity and boundedness of the datum $f$. It would be interesting to see if our results still hold with weaker assumptions on $f$, such as
$$\exists E_1,E_2\subset D,\ |E_1|,|E_2|>0\ :\ f\,\big|_{E_1}>0\quad\text{and}\quad f\,\big|_{E_2}<0,$$
or some $L^p$-norm of $f$ is bounded while the $\infty$-norm is not. A similar question for Dirichlet boundary conditions has been recently considered in \cite{butshr} by constructing right-hand sides $f\in W^{-1,p'}(D)\cap L^1(D)$ such that the optimal set $\bar{\Om}$ can be any $p$-quasi open set $\Om\subset D$. We expect a similar behavior in our case, even if some new technical difficulties arise due to the fact that solutions of a relaxed Robin problem of the form \eqref{functionalrelaxed} in a general domain $\Om$ may have internal discontinuities.

%; however, in our case the same strategy does not work because if we fix $\Om_0$ and consider the corresponding torsion function $w_0$ starting from the functional
%$$\int_{D \cap \{u \neq 0\}} |\nabla u|^2 \,dx + \int_{D \cap \{u \neq 0\}} (- f(x)u + 1) \, dx + \int_{J_u} \left[ |u^+|^2 + |u^-|^2\right] \, d\cH^{d-1}.$$
%The problem is that the domain corresponding to $w_0$ may not coincide with $\Om_0$ because, a priori, internal fractures could appear.

%%%%%%%%%%%%%%%%%%%%%%%%%%%%%%%%%%%%%%%%%%%%%%%%%%%%%%%%%%%%
\section*{Acknowledgements} The work of GB is part of the project 2017TEXA3H {\it``Gradient flows, Optimal Transport and Metric Measure Structures} funded by the Italian Ministry of Research and University.

\bigskip
{\small\noindent
Giuseppe Buttazzo:
Dipartimento di Matematica, Universit\`a di Pisa\\
Largo B. Pontecorvo 5, 56127 Pisa - ITALY\\
{\tt giuseppe.buttazzo@unipi.it}\\
{\tt http://www.dm.unipi.it/pages/buttazzo/}

\bigskip\noindent
Francesco Paolo Maiale:
Scuola Normale Superiore\\
Piazza dei Cavalieri 7, 56126 Pisa - ITALY\\
{\tt francesco.maiale@sns.it}\\
{\tt https://poisson.phc.dm.unipi.it/\char`\~fpmaiale/}

\end{document}